\documentclass[reqno]{amsart}
\usepackage[utf8]{inputenc}
\usepackage[USenglish]{babel}
\usepackage{amsmath,amsfonts,amssymb,amsthm,dsfont,bm}
\usepackage{enumitem}
\usepackage[usenames,dvipsnames]{xcolor}
\usepackage[colorlinks=true]{hyperref} 
\usepackage{verbatim}
\usepackage{graphicx}
\usepackage{overpic}
\usepackage{mathabx}
\usepackage{cite}
\hypersetup{linkcolor=Violet,citecolor=PineGreen}
\newtheorem{thm}{Theorem}[section]
\newtheorem{lem}[thm]{Lemma}
\newtheorem{prop}[thm]{Proposition}
\newtheorem{cor}[thm]{Corollary}
\theoremstyle{definition}
\newtheorem{defn}[thm]{Definition}
\newtheorem{eje}[thm]{Example}
\newtheorem{rmk}[thm]{Remark}
\newtheorem{rmks}[thm]{Remarks}
\newtheorem{assu}[thm]{Assumption}
\numberwithin{equation}{section}
\definecolor{col}{rgb}{0,0,0.6}
\newcommand{\N}{\mathbb{N}}
\newcommand{\R}{\mathbb{R}}

\newcommand{\M}{\mathcal{M}}

\newcommand{\Om}{\mathcal{O}}
\newcommand{\A}{\mathbb{A}}

\newcommand{\dd}{\textsf{d}}

\newcommand{\ep}{\varepsilon}
\newcommand{\des}{\displaystyle}

\makeatletter
\@namedef{subjclassname@2020}{%
  \textup{2020} Mathematics Subject Classification}

\makeatother
\begin{document}
\title[]
{Tracking nonautonomous attractors in singularly perturbed systems of ODEs with dependence on the fast time}
\author[I.P.~Longo]{Iacopo P. Longo}
\author[R.~Obaya]{Rafael Obaya}
\author[A.M.~Sanz]{Ana M. Sanz}
\address[I.P. Longo]{Department of Mathematics, Imperial College London, 635 Huxley Building, 180 Queen's Gate South Kensington Campus, London SW7 2AZ, United Kingdom.} \email{iacopo.longo@imperial.ac.uk}
\address[R. Obaya]{Departamento de Matem\'{a}tica
Aplicada, Escuela de Ingenier\'{\i}as Indus\-tria\-les (Sede Doctor Mergelina), Universidad de Valladolid,
47011 Valladolid, Spain, and member of IMUVA, Instituto de Investigaci\'{o}n en
Matem\'{a}ticas, Universidad de Va\-lla\-dolid, Spain.}
 \email{rafael.obaya@uva.es}
\address[A.M. Sanz]{Departamento de Did\'{a}ctica de las Ciencias Experimentales, Sociales y de la Matem\'{a}tica,
Facultad de Educaci\'{o}n y Trabajo Social, Universidad de Valladolid, 47011 Valladolid, Spain,
and member of IMUVA, Instituto de Investigaci\'{o}n en  Mate\-m\'{a}\-ti\-cas, Universidad de
Valladolid, Spain.} \email{anamaria.sanz@uva.es}
\thanks{All authors were partly supported by MICIIN/FEDER project
PID2021-125446NB-I00 and by the University of Valladolid under project PIP-TCESC-2020. I.P.~Longo was also partly supported  by UKRI under the grant agreement EP/X027651/1.}
\date{}
\begin{abstract}
New results on the behaviour of the fast motion in slow-fast systems of ODEs with dependence on the fast time are given in terms of tracking of nonautonomous attractors. Under quite general assumptions, including the uniform ultimate boundedness of the solutions of the layer problems, inflated pullback attractors are considered. In general, one cannot disregard the inflated version of the pullback attractor, but it is possible under the continuity of the fiber projection map of the attractor.
The problem of the limit of the solutions of the slow-fast system at each fixed positive value of the slow time is also treated and
in this formulation the critical set is given by the union of the fibers of the pullback attractors. The results can be seen as extensions of the classical Tikhonov theorem to the nonautonomous setting.
\end{abstract}
\maketitle
\section{Introduction}\label{secintro}
The paper analyses singularly perturbed systems of ODEs where the fast equation exhibits dependence on the fast time. Namely, the systems take the form
	\begin{equation*}
\left\{\begin{array}{l} \dot{x}  = f(x,y)\,,
  \\[0.1cm]
 \ep\,\dot{y}  = g(x,y,t/\ep )\,,
\end{array}\right.
\end{equation*}
where the phase variables $(x,y)\in\R^n\times\R^m$ respectively denote the slow and fast variables of the problem,  $t$ is the slow time and  $\ep>0$ is a small parameter. Taking the change of scale $\tau =t/\ep$, then $\tau$  denotes the fast time and the previous equations become
\begin{equation*}
\left\{\begin{array}{l} x'  = \ep\,f(x,y)\,,
  \\
 y' = g(x,y,\tau )\,.
\end{array}\right.
\end{equation*}
Initial conditions $x(0) = x_0$, $y(0) = y_0$  and a positive  value of time $t_0>0$ in the slow timescale are fixed. In this work we apply methods of nonautonomous dynamical systems to investigate the  behaviour  of the fast motion $y_\ep(\tau)$ for $\tau \in (0,t_0/\ep]$ when the parameter $\ep$ is small,  assuming natural and simple  hypotheses on the slow-fast systems above. In order to motivate and explain the content of the paper, we briefly describe some previous important results in the literature.

We begin by considering the case in which the slow-fast system is autonomous, i.e., the map $g$ is independent of $\tau$. A crucial problem in  singular perturbation theory is to identify the limits of the solutions $(x_\ep(t),y_\ep(t))$ at each fixed value of the slow time $t\in (0,t_0]$ as $\ep\to 0$. A remarkable answer to this question was given by Tikhonov~\cite{paper:Tikh} in terms of the reduced order system $\dot{x}=f(x,y)$, $0=g(x,y)$. Assuming that the algebraic equation can be inverted (at least locally) providing  $y=\phi(x)$ which defines a manifold of  (locally) asymptotically stable fixed points of the layer equations $\dot{y}=g(x,y)$,  uniformly in $x$, and that the initial datum $y_0$ lies in the domain of attraction of the equilibrium $\phi(x_0)$ for the layer equation for $x_0$, Tikhonov's result asserts that $x_\ep(t)$ converges uniformly for $t \in [0,t_0]$  to $x_0(t)$, the solution of the slow equation  $\dot{x}=f(x,\phi(x))$, $x(0)=x_0$. Moreover, the solutions $y_\ep(t)$ converge   to $\phi(x_0(t))$ for $t \in (0,t_0]$. In this formulation  the set $C=\{(x,y)\mid g(x,y)=0\}$  is called the critical manifold, or the critical set if it is not smooth.  A detailed exposition of this theory and some other interesting results that extend it, can be found in O'Malley~\cite{book:OMalley}, Tikhonov et al.~\cite{book:TVS} and Wasow~\cite{book:Wasow}. Later, assuming that the critical points $(x,\phi(x))$ are hyperbolic, i.e., all real parts of the eigenvalues of the Jacobian matrix  $\partial_y g( x, \phi(x))$ are non-zero, Fenichel~\cite{paper:FE79} provides  a more detailed description of the dynamics of the slow-fast systems, proving the existence of normally hyperbolic invariant manifolds  in a neighbourhood of the critical manifold when the parameter $\ep$ is small enough. An introduction to the combined applications of the  Tikhonov and Fenichel results can be found in Berglund and Gentz~\cite{book:BG}, Kuehn~\cite{book:ku} and Verhulst~\cite{book:Ve}.

The first relevant references (up to our knowledge) dealing with the limit problem as $\ep\to 0$ of the solutions $(x_\ep(t),y_\ep(t))$ valid for  slow-fast systems with dependence on the fast time are Artstein and Vigodner~\cite{paper:AV96} and Artstein~\cite{paper:ZA99}. They  formulate  the theory in terms of Carath\'{e}odory dependence of $g$ on $\tau$  and  mild conditions that imply that $\mathcal{H}$, the hull of $g$, defined by the closure for a weak topology of the set of translated maps $ (g{\cdot}s)(x,y,\tau)=g(x,y,s+\tau)$, $s\in\R$ is metric and compact. The most important hypothesis required in these papers is a kind of uniform boundedness of the solutions of the layer equations, which end up being included in balls that remain uniformly bounded for $x$ in compact subsets of $\R^n$. They prove that the solutions $x_\ep(t)$ converge  to the compact set of solutions of the  multivalued differential equations $\dot{x}\in F(x):=\big\{ \int f(x,y)d\mu(y) \mid \mu \; x\hbox{-invariant}\big\}$, $x(0)=x_0$. In fact,  the invariant measures of the layer equations $y'=g(x,y,\tau)$ for each $x$ provide a new approach to the critical set in this nonautonomous setting. Precisely, the description of the behaviour of the fast variables is given in terms of statistical convergence, which means convergence of the measures associated to the solutions $y_\ep(t)$ to a precise set of measures, in a weak topology. This has been explained in more detail at the beginning of Section~\ref{secTracking}.

In this paper, for the sake of simplicity, we assume that $g$ is bounded and uniformly continuous on  $K\times \R$ for compact sets $K \subset \R^{n+m}$,  which implies that $\mathcal{H}$, the hull of $g$,  is compact for the compact-open metric topology and the translation flow $(\tau,h)\in\R\times\mathcal{H} \mapsto h{\cdot}\tau\in\mathcal{H}$ is continuous. In the first part of the work,  we assume that the solutions of the $x$-parametric families over the hull are globally defined in the future.  We reformulate the boundedness condition in~\cite{paper:AV96} and~\cite{paper:ZA99} requiring the solutions of  the layer equations $y'=g(x,y,\tau)$ for $x \in K_0$ to be uniformly ultimately bounded, for each compact subset $K_0 \subset \R^n$. As a first consequence of this  assumption, we  deduce that the global skew-product semiflow  on $\mathcal{H}\times K_0 \times \R^m$ induced by the solutions, which can be called {\it the layered semiflow},  admits a global (pullback) attractor $\A=\bigcup_{(h,x)\in \mathcal{H}\times K_0} \{(h,x)\}\times A_{(h,x)} \subset \mathcal{H}\times K_0 \times \R^m$ (see Kloeden and Rasmussen~\cite{klra}). Aiming at obtaining properties of forwards convergence of the solutions,  we introduce  the parametrically inflated pullback attractor $\{A_{(h,x)}[\delta]\}_{(h,x)\in \mathcal{H}\times K_0}$ ($\delta>0$) in the terms of Wang et al. \cite{walikl}. The  main conclusion of this part is that the solutions $y_\ep(\tau)$ track the fibers of the inflated pullback attractor when the parameter $\ep$ is small and $\tau \in [T,t_0/\ep]$ for a certain $T>0$. More precisely, if    $x_{\ep_k}(t)$ converges to $x_0(t)$ uniformly  on $[0,t_0]$ for a sequence $\ep_k \rightarrow 0$, for a tubular compact set $K_0\subset \R^n$ around the slow limit $x_0(t)$, fixed a $\delta>0$, there exist $T>0$ and $k_0>0$ such that $2T<t_0/\ep_{k_0}$ and for all $k\geq k_0$,  ${\rm dist}\big(y_{\ep_k}(\tau),A_{(g{\cdot}\tau,x_0(\ep_k\tau))}[\delta]\big)\leq  \delta$ for every $\tau \in [T,t_0/\ep_k]$. In addition, if $g{\cdot}(t_0/\ep_k)\to h_0$ in $(\mathcal{H},d)$ as $k \to \infty$, then $\lim_{k\to \infty}  {\rm dist}\big(y_{\ep_k}(t_0/\ep_k),A_{(h_0,x_0(t_0))}\big) =0$, which offers an answer to the question on the limit of the solutions of the slow-fast systems at a fixed time. An example is given to show that in general the inflated version of the pullback attractor is indispensable, but as a relevant second result we prove that when the compact fibers $A_{(h,x)}$ of the pullback attractor vary continuously  with respect to $(h,x) \in \mathcal{H}\times K_0$ for the Hausdorff metric, then we can  avoid the use of the parametrically inflated pullback attractor, since then
the fast motion tracks the fibers of the pullback attractor itself.

In some sense we can say that these results give a nonautonomous version of the Tikhonov  theorem. Note that although the classical theory requires asymptotically stable equilibria, uniformly in $x$, in the nonautonomous case the process given by $g$ and each $x$ has a pullback attractor which in principle has no properties of forward attraction. However, the counterpart of the critical set in our approach relies on the union of the fibers of the pullback attractor of the layered semiflow, $\{(x,\bigcup_{h\in\mathcal{H}}A_{(h,x)})\mid x\in K_0\}$, and the set $\mathcal{A} :=\bigcup_{(h,x)\in\mathcal{H}\times K_0}A_{(h,x)}$ is the {\it uniform attractor} for the layered semiflow, and it enjoys some forward attraction properties (see, e.g., Carvalho et al.~\cite{book:CLR}). The internal dynamics of this critical set can be (locally or globally) uniformly stable, sensitive  with respect to initial conditions, chaotic or even a transition between them, and it evolves with the fast time. For this reason, in this nonautonomous formulation,  the term {\it slow manifold} for some invariant subset  close to this critical set seems not to be appropriate.

In the second part of the work, the general assumptions on the existence and  boundedness of the solutions of the layer equations are relaxed. Namely, the layered semiflow is defined with base on $\Om(g)\times K_0$ for the omega-limit set of $g$ and a compact set $K_0$ of $\R^n$, and the uniform ultimate boundedness of solutions of the layer equations on  compact sets $K_0$ is required on $[0,\infty)$. We check that these conditions are equivalent to the hypotheses in the papers~\cite{paper:AV96} and~\cite{paper:ZA99} ,  and in addition they imply the assumptions in the first part when $g\in \Om(g)$. In any case, we can adapt the methods and conclusions of the previous part to this more general setting. As it turns out, the solutions $y_\ep(\tau)$ track the fibers of the parametrically  inflated pullback attractor along pieces of orbits inside  $\Om(g)$. As before,  when $A_{(h,x)}$ varies continuously with respect to $(h,x)\in \Om(g)\times K_0$, the fast motion tracks the fibers of the  pullback attractor and the inflated version is not required. The same comments above regarding the limit of $(x_\ep(t_0), y_\ep(t_0))$ and the critical set are equally correct after including the obvious changes required in this setting.

The paper is organised in five sections. Section~\ref{secpreli} collects some preliminaries  in nonautonomous dynamics. Sections~\ref{secSlowfast} and~\ref{secTracking} correspond to the aforementioned first part of the paper, exploiting the imposed conditions in Section~\ref{secSlowfast} and proving the main results in Section~\ref{secTracking}, together with some numerical simulations to illustrate the theoretical results in Subsection~\ref{sec:subsec numeric}. Finally, Section~\ref{secOmegalim} corresponds to the second part of the work, where the results are presented under relaxed assumptions.
\section{Preliminaries}\label{secpreli}
In this section we include some preliminaries on
topological dynamics and different types of attractors for nonautonomous dynamical systems.

Let $(P,d)$ be a compact
metric space. A real {\em continuous flow\/} $(P,\theta,\R)$ is a dynamical system given by  a continuous map $\theta: \R\times P \to  P,\;
(t,p)\mapsto \theta(t,p)=\theta_t(p)$ which satisfies $\theta_0=\text{Id}$ and $\theta_{t+s}=\theta_t\circ\theta_s$ for each $s$, $t\in\R$.
The set $\{ \theta_t(p) \mid t\in\R\}$ is the {\em orbit\/}
of the point $p$. We say that a subset $P_1\subset P$ is {\em
$\theta$-invariant\/} if $\theta_t(P_1)=P_1$ for every $t\in\R$.
The flow $(P,\theta,\R)$ is {\em minimal\/} if $P$ does not contain properly any other
compact $\theta$-invariant set; equivalently,  if every
orbit is dense. The flow is  {\em almost periodic\/} if the family of maps $\{\theta_t\}_{t\in \R}:P\to P$ is uniformly equicontinuous.
\par
A finite regular measure defined on the Borel sets of $P$ is called
a Borel measure on $P$. A normalized Borel measure $\mu$  on
$P$ is {\em invariant\/} for the flow $(P,\theta,\R)$, or it  is {\em $\theta$-invariant\/} if $\mu(\theta_t(P_1))=\mu(P_1)$ for every Borel subset
$P_1\subset P$ and every $t\in \R$.  If, in
addition, $\mu(P_1)=0$ or $\mu(P_1)=1$ for every
$\theta$-invariant Borel subset $P_1\subset P$, then it is {\em ergodic\/}.
$(P,\theta,\R)$ is {\em uniquely ergodic\/} if it has a
unique normalized invariant measure, which is then necessarily
ergodic. A minimal and almost periodic flow $(P,\theta,\R)$ is uniquely ergodic.
\par
Let $\R_+=\{t\in\R\,|\,t\geq 0\}$. Given a continuous compact flow $(P,\theta,\R)$ and a
complete metric space $(X,\dd)$, a continuous {\em skew-product semiflow\/} $(P\times
X,\pi,\,\R_+)$ on the product space $P\times X$ is determined by a continuous map
\begin{equation*}
 \begin{array}{cccl}
 \pi \colon &  \R_+\times P\times X& \longrightarrow & P\times X \\
 &(t,p,x) & \mapsto &(p{\cdot}t,u(t,p,x))
\end{array}
\end{equation*}
 which preserves the flow on $P$, called the {\em base flow\/} and denoted by $\theta_t(p)=p{\cdot}t$ for simplicity.
 The semiflow property means:
\begin{enumerate}
\renewcommand{\labelenumi}{(\roman{enumi})}
\item $\pi_0=\text{Id},$
\item $\pi_{t+s}=\pi_t \circ \pi_s$ for all  $t$, $s\geq 0\,,$
\end{enumerate}
where $\pi_t(p,x)=\pi(t,p,x)$ for each $(p,x) \in P\times X$ and $t\in \R_+$.
This leads to the so-called {\em semicocycle\/} property of the {\em cocycle mapping\/} $u$,
\begin{equation*}
 u(t+s,p,x)=u(t,p{\cdot}s,u(s,p,x))\quad\mbox{for $s,t\ge 0$ and $(p,x)\in P\times X$}.
\end{equation*}
\par
The set $\{ \pi(t,p,x)\mid t\geq 0\}$ is the {\em semiorbit\/} of
the point $(p,x)$. A subset  $K$ of $P\times X$ is {\em positively
invariant\/} if $\pi_t(K)\subseteq K$
for all $t\geq 0$ and it is $\pi$-{\em invariant\/} if $\pi_t(K)= K$
for all $t\geq 0$.  A compact $\pi$-invariant set $K$ for the
semiflow  is {\em minimal\/} if it does not contain any nonempty
compact $\pi$-invariant set  other than itself.
\par
If $K\subset P\times X$ is a compact positively invariant set, we denote by $\mathcal{M}_{\rm inv}(K)$ the set of normalized Borel measures $\mu$ on $K$ which are $\pi$-invariant, that is, such that $\mu(B)=\mu(\pi_t^{-1}(B))$ for every Borel set $B$ in $K$ and  $t\geq 0$.

The reader can find in  Ellis~\cite{elli}, Shen and Yi~\cite{shyi} and references therein, a
more in-depth survey on topological dynamics.
\par
Since nonautonomous attractors play a crucial role in this paper, we review some basic facts. Some good references are the books by Carvalho et al.~\cite{book:CLR} and Kloeden and Rasmussen~\cite{klra}. For the sake of completeness, first of all we recall the definition and some properties of the Hausdorff semidistance between subsets of a metric space $(Y,d)$. Given $Y_1,Y_2\subset Y$, the Hausdorff semidistance is defined by
\begin{equation*}
  {\rm dist}(Y_1,Y_2) := \sup_{y_1\in Y_1}\inf_{y_2\in Y_2} d(y_1,y_2)=\sup_{y_1\in Y_1} {\rm dist}(y_1,Y_2)\,,
\end{equation*}
where ${\rm dist}(y_1,Y_2)$ denotes the usual distance between the point $y_1$ and the set $Y_2$, with no conflict in the notation when $Y_1$ is a singleton. Sometimes in the literature it is written ${\rm dist}_Y$ to indicate in which space we are, but we will omit this fact. The Hausdorff metric between two compact subsets $Y_1,Y_2$ of $Y$ is  defined by
\[
d_{H}(Y_1,Y_2) := \max\{{\rm dist}(Y_1,Y_2),{\rm dist}(Y_2,Y_1)\}\,.
\]
We will be using that, if $Y_1\subset Z_1$, then ${\rm dist}(Y_1,Y_2)\leq {\rm dist}(Z_1,Y_2)$; and if $Z_2\subset Y_2$, then ${\rm dist}(Y_1,Y_2)\leq {\rm dist}(Y_1,Z_2)$, for subsets $Y_1, Y_2, Z_1, Z_2$ of $Y$.

For a skew-product semiflow $\pi$ on $P\times X$ with a compact base $P$, {\em the global attractor\/}   $\A\subset P\times X$, when it exists,  is an invariant compact set attracting a certain class of sets in $P\times X$ forwards in time; namely,
\[
\lim_{t\to\infty} {\rm dist}\big(\pi_t(P\times X_1),\A\big)=0 \quad\text{for each}\; X_1\in \mathcal{D}(X)\,,
\]
for the Hausdorff semidistance ${\rm dist}$. The standard  choices for $\mathcal{D}(X)$ are either the class of bounded subsets $\mathcal{D}_b(X)$ of $X$ or that of compact subsets $\mathcal{D}_c(X)$ (see Cheban et al.~\cite{chks}).
Since $P$ is compact,  the non-autonomous set $\{A_p\}_{p\in P}$, formed by the $p$-sections of $\A$ defined by $A_p=\{x\in X\mid (p,x)\in \A\}$ for each $p\in P$, is a {\em pullback attractor\/}, that is,  $\{A_p\}_{p\in P}$ is compact, invariant, and it pullback attracts all the sets $X_1 \in \mathcal{D}(X)$:
\begin{equation*}
\lim_{t\to\infty} {\rm dist}\big(u(t,p{\cdot}(-t),X_1),A_p\big)=0\quad \text{for all}\; p\in P\,.
\end{equation*}
The implications for each fixed $p\in P$ are expressed in terms of {\em pullback attraction for the related evolution process on $X$\/}, defined by $S_{p}(t,s)\,x=u(t-s,p{\cdot}s,x)$ for each $x\in X$ and $t\geq s$. Precisely,
for each fixed $p\in P$, the family of compact sets $\{A_{p{\cdot}t}\}_{t\in  \R}$ is the pullback attractor for the process  $S_{p}(\cdot,\cdot)$,
 meaning that:
\begin{itemize}
\item[(i)] it is invariant, i.e.,  $S_{p}(t,s)\,A_{p{\cdot}s}=A_{p{\cdot}t}$ for all $t\geq s\,$;
\item[(ii)] it pullback attracts the class of sets $\mathcal{D}(X)$, i.e., for each  $X_1\in \mathcal{D}(X)$,
\[
\lim_{s\to -\infty}{\rm dist}\big(S_{p}(t,s)\,X_1,A_{p{\cdot}t}\big)=0 \quad \hbox{for all}\; t\in \R\,;
\]
\item[(iii)] it is the minimal family of closed subsets of $X$ with property~(ii).
\end{itemize}
However, the forward and pullback dynamics are in general unrelated, so that it is always an interesting  problem to know whether the pullback attractor is also a {\em forward attractor\/} for the process. That is to know if, for each $p\in P$,
\begin{equation*}
\lim_{t\to \infty}{\rm dist}\big(u(t,p,X_1),A_{p{\cdot}t}\big)=0   \quad \hbox{for each}\; X_1\in \mathcal{D}(X)\,.
\end{equation*}

Finally, let us fix some notation and recall some definitions. The norm of a vector $v\in\R^N$ is denoted by $|v|$. We say that $\{v_k\}_{k\geq 1}$ converges to a set $C\subset \R^N$ if ${\rm dist}(v_k,C)\to 0$ as $k\to\infty$. If the set $C$ is compact, it is easy to check that $\{v_k\}_{k\geq 1}$ converges to $C$ if and only if for every convergent subsequence $\{v_{k_j}\}_{j\geq 1}$, $\lim_{j\to\infty} v_{k_j}\in C$. A solution to a differential inclusion $\dot{x}\in F(x)\subset\R^N$ on an interval $[0,t_0]$ is an absolutely continuous function $x:[0,t_0]\to\R^N$  such that $x(t)\in F(x(t))$ for almost every $t$ in $[0,t_0]$. Given a set $\mathcal{F}\subset C([0,t_0],\R^N)$, we say that a family of continuous maps $\{x_\ep(t)\mid 0<\ep\leq \ep_0\}$ converge to $\mathcal{F}$ if ${\rm dist}(x_\ep,\mathcal{F})=\inf_{x\in \mathcal{F}}\|x_\ep-x\|_\infty\to 0$ as $\ep\to 0$, where $\|\cdot\|_\infty$ denotes the sup-norm on $[0,t_0]$. As before, if $\mathcal{F}$ is compact, the limit points of converging sequences $\{x_{\ep_k}\}_{k\geq 1}$ are in $\mathcal{F}$.
\section{Slow-fast systems with nonautonomous fast dynamics: terminology, conditions and some basic results}\label{secSlowfast}
We consider a coupled pair of ordinary differential equations, the second of which is singularly perturbed and shows an explicit dependence on time:
\begin{equation}\label{eq:slowfastent}
\left\{\begin{array}{l} \dot{x}  = f(x,y)\,,
  \\[0.2cm]
  \ep\,\dot{y}  = g\left(x,y,\des\frac{t}{\ep} \right),
\end{array}\right.
\end{equation}
where $\ep>0$ is a small parameter. With this structure, there are two different timescales which are well-separated, so that the system is also called a slow-fast system of ODEs.
The components of $x\in \R^n$ are called slow variables and the components of $y\in \R^m$ are called fast variables. Here, $\dot{x}$ and $\dot{y}$ denote the derivatives with respect to the time $t$. The timescales are denoted by $t$ the slow time and by $\tau$ the fast time, both timescales being related by means of $\ep$ through the formula $\tau=t/\ep$. When initial conditions are prescribed,
\begin{equation}\label{eq:initial values}
\left\{\begin{array}{l} x(0)=x_0\,,
  \\
  y(0)=y_0\,,
\end{array}\right.
\end{equation}
the interest is in the study of the limit behaviour as $\ep\to 0$ of the solutions  $(x_\ep(t),y_\ep(t))$ of the initial value problems~\eqref{eq:slowfastent}-\eqref{eq:initial values} for values of $t$ on a compact interval, say $[0,t_0]$.

Equivalently, we can write the system in the fast scale of time $\tau$, denoting by $x'$ and $y'$ the derivatives with respect to $\tau$:
\begin{equation}\label{eq:slowfastentau}
\left\{\begin{array}{l} x'  = \ep\,f(x,y)\,,
  \\
 y' = g(x,y,\tau )\,,
\end{array}\right.
\end{equation}
where the time dependence on the fast time scale is evident in the fast variables. Now, the interest is in the limit behaviour as $\ep\to 0$ of the solutions  $(x_\ep(\tau),y_\ep(\tau))$ for $\tau$ on growing intervals of the form $[0,t_0/\ep]$, and specifically in the behaviour of $y_\ep(t_0/\ep)$ as $\ep\to 0$. It will be useful to see the fast dynamics as a perturbation of the parameter-dependent ODEs $y'= g(\lambda,y,\tau )$, in which the parameter $\lambda$ slowly varies in time.

These are the initial basic assumptions we make on the maps $f$ and $g$ in~\eqref{eq:slowfastentau} so as to have existence of solutions for the slow-fast pair, global existence for the slow variables,   and uniqueness of solutions and continuous dependence with respect to $x$ for the fast dynamics.
\begin{assu}\label{asu:1}
\begin{itemize}
  \item[]
  \item[(i)] $f(x,y)$ is continuous and
  \begin{equation*}
    \sup_{y\in \R^m} |f(x,y)|\leq a + b\,|x|
  \end{equation*}
  for all $x\in\R^n$, for some  $a$ and $b$ positive.
  \item[(ii)] $g(x,y,\tau)$ is an {\em admissible\/} map, that is, for every
compact set $K\subset \R^n\times\R^m$, $g$ is bounded and uniformly continuous
on $K\times \R$.
   \item[(iii)] Given a compact set $K_0\subset \R^n$, $g$ is uniformly  Lipschitz in $y$ on compact subsets $K\subset \R^m$, that is, given a compact set $K\subset \R^m$ there is an $L=L(K_0,K)$ such that
      \[
      |g(x,y_1,\tau)-g(x,y_2,\tau)|\leq L\,|y_1-y_2|
      \]
      for every $y_1,y_2\in K$, $x\in K_0$ and $\tau\in \R$.
\end{itemize}
\end{assu}
The growth condition in (i) is taken from Assumption~5.1~(ii) in~\cite{paper:ZA99} and it is a technical condition frequently found in the theory of multivalued differential equations (see Deimling~\cite{book:Dei}). Note that it implies, with an application of Gronwall's inequality,  that for $\ep>0$, subject to the existence of the solutions,  the components $x_\ep(t)$ of the solution $(x_\ep(t),y_\ep(t))$ of  the  problem~\eqref{eq:slowfastent}-\eqref{eq:initial values} satisfy $|x_\ep(t)|\leq e^{bt}|x_0|+ a(e^{bt}-1)/b$ for every $t\geq 0$, and in particular they are uniformly bounded on compact intervals $[0,t_0]$. In fact, one can usually disregard this growth condition. To this respect, see Remark~\ref{rmk:growth bound}~(1).

Condition (ii) is imposed in order to build a compact base flow out of the vector field $g$ by means of the so-called {\em hull\/} construction. Let us explain this in detail. Provided that $g$ is admissible, its {\em hull\/} $\mathcal{H}$, which is  the closure for the compact-open topology of the set of $s$-translates of $g$, $\{ g{\cdot}s \mid s\in\R\}$ with $(g{\cdot}s)(x,y,\tau)=g(x,y,\tau+s)$
for $\tau\in \R$ and $(x,y)\in\R^n\times\R^m$, is a compact metric space. Let us denote by $d$ the metric on $\mathcal{H}$.
 In this situation, the translation or shift map $\R\times \mathcal{H}\to \mathcal{H}$, $(\tau,h)\mapsto h{\cdot}\tau$ defines a
continuous flow on the compact space $\mathcal{H}$.  The shift flow is  minimal  provided that  the map $g$ has certain recurrent behaviour in time, such as periodicity, almost periodicity, or other weaker properties of recurrence. Since the time variation will be almost periodic in some forthcoming examples, we recall that the map $g(x,y,\tau)$ is said to be {\it almost periodic, uniformly on compact sets\/}, if it is admissible and almost periodic in $\tau$ for each fixed $(x,y)$. If so, then the flow on the hull is minimal and almost periodic, and thus uniquely ergodic.

Let us introduce some notation for the solutions of the nonautonomous ODEs for the fast motion, also called {\it layer problems} or  {\it layer equations},
\begin{equation*}
  y' = g(x,y,\tau)\,,
\end{equation*}
where $x$ is seen as a parameter. For a fixed $x\in\R^n$, the solution  such that $y(\tau_0)=y_0$ is denoted by $y(\tau,g,x,\tau_0,y_0)$, for $\tau$ in the interval of existence.   Analogously, given $h\in\mathcal{H}$,   $y(\tau,h,x,\tau_0,y_0)$ denotes the solution of $y' = h(x,y,\tau)$ for the value of the parameter $x$ and with initial condition $y(\tau_0)=y_0$. When $\tau_0=0$, we remove it from the notation, so that $y(\tau,h,x,y_0)$ is the solution of $y' = h(x,y,\tau)$ with initial condition $y(0)=y_0$.

Note that, because of the time dependence in the equations, when a solution is translated, it is a solution of a translated equation. In  particular, using the language of processes (see Section~\ref{secpreli}) for each $h\in\mathcal{H}$, $x\in\R^n$, $y_0\in\R^m$, for $\tau\geq \tau_0$, as far as the solutions are defined,
\begin{equation}\label{eq:translated}
 S_h^x(\tau,\tau_0)\,y_0=y(\tau-\tau_0,h{\cdot}\tau_0,x,y_0) = y(\tau,h,x,\tau_0,y_0)\,.
\end{equation}

Now we build a skew-product semiflow, in principle only locally defined, induced by the solutions of the systems over $\mathcal{H}$, the hull  of $g$. Consider a compact set of parameters $K_0\subset \R^n$, and on the compact product space $\mathcal{H} \times K_0$ define the continuous flow given by
the map
\begin{equation*}
 \begin{array}{ccl}
 \theta:\R\times\mathcal{H} \times K_0& \longrightarrow & \mathcal{H} \times K_0 \\
 (\tau,h,x) & \mapsto &(h{\cdot}\tau,x)\,.
\end{array}
\end{equation*}
Then, under Assumption~\ref{asu:1} a continuous skew-product semiflow is defined over the former base flow in the following way, on an appropriate open set $\mathcal{U}$ subject to the existence of the solutions:
\begin{equation}\label{eq:sk-pr}
 \begin{array}{ccl}
 \pi:\mathcal{U}\subseteq\R_+\times\mathcal{H} \times K_0 \times \R^m & \longrightarrow & \mathcal{H} \times K_0 \times \R^m \\
 (\tau,h,x,y_0) & \mapsto &(h{\cdot}\tau,x,y(\tau,h,x,y_0))\,.
\end{array}
\end{equation}
Due to its structure, we can call $\pi$ the {\em layered semiflow\/}.
Note that the admissibility condition (ii) is used for the joint continuity and the fiber components satisfy the so-called {\it cocycle property}, subject to existence:
\begin{equation}\label{eq:cocycle}
  y(\tau_1+ \tau_2,h,x,y_0)=y(\tau_1,h{\cdot}{\tau_2},x,y(\tau_2,h,x,y_0))\,,\quad \tau_1,\tau_2\geq 0\,.
\end{equation}

Note also that for each fixed $x\in \R^n$  a continuous skew-product semiflow is defined over the base flow on the hull $\mathcal{H}$ in the following way:
\begin{equation}\label{eq:sk-pr x}
 \begin{array}{rccl}
 \pi^x: &\mathcal{U}_x\subseteq \R_+\times\mathcal{H}  \times \R^m & \longrightarrow & \mathcal{H}  \times \R^m \\
 &(\tau,h,y_0) & \mapsto &(h{\cdot}\tau,y(\tau,h,x,y_0))\,,
\end{array}
\end{equation}
on an appropriate open set $\mathcal{U}_x$ subject to the existence of the solutions.

Next, we specify the conditions required on the fast equations in Section~\ref{secTracking}, and we explore some initial consequences. These conditions will be later relaxed in  Section~\ref{secOmegalim} to obtain a more general version of the results.

 The main assumption is the uniform ultimate boundedness of the solutions of the parametric families $y'=g(x,y,\tau)$ for $x$ in compact sets $K_0\subset \R^n$. The reader can see, e.g., Definition~5.1 in Longo et al.~\cite{paper:LNO2}, given in terms of processes.
\begin{defn}\label{def:ultimately}
Let $K_0\subset \R^n$ and assume that for each $x\in K_0$, $\tau_0\in \R$ and $y_0\in \R^m$  the solution  $y(\tau,g,x,\tau_0,y_0)$ is defined on $[\tau_0,\infty)$.  The solutions of the parametric family $y'=g(x,y,\tau)$ for $x\in K_0$ are said to be {\em uniformly ultimately bounded\/} if there is a constant $c>0$ such that for every $d>0$ there is a time $T=T(d)$ such that
\begin{equation*}
  |S_g^x(\tau+\tau_0,\tau_0)\,y_0|=|y(\tau+\tau_0,g,x,\tau_0,y_0)|\leq c
\end{equation*}
for every $x\in K_0$, whenever  $\tau_0\in \R$, $\tau\geq T$ and $|y_0|\leq d$.
\end{defn}

\begin{assu}\label{asu:2}
\begin{itemize}
  \item[]
  \item[(i)] For each $h\in\mathcal{H}$, $x\in \R^n$ and $y_0\in \R^m$, the solution $y(\tau,h,x,y_0)$ is defined for all $\tau\geq 0$.
  \item[(ii)] For each compact set $K_0$ of $\R^n$, the solutions of the parametric family $y'=g(x,y,\tau)$ for $x\in K_0$ are uniformly ultimately bounded.
\end{itemize}
\end{assu}
Note that, by relation~\eqref{eq:translated}, condition (i) implies that the solution $y(\tau,h,x,\tau_0,y_0)$ is defined for all $\tau\geq \tau_0$. In particular,  we are assuming that all the solutions of the layer equations on the hull are globally defined in the future, so that whenever we take a compact set $K_0\subset \R^n$, the skew-product semiflow~\eqref{eq:sk-pr} is defined on  $\R_+\times\mathcal{H} \times K_0 \times \R^m$ as well as all the skew-product semiflows~\eqref{eq:sk-pr x} for $x\in \R^n$ are defined on  $\R_+\times\mathcal{H} \times \R^m$.

The first result is on the existence of a global attractor for the induced skew-product semiflows.
\begin{prop}\label{prop:existe atractor}
Under Assumptions~{\rm\ref{asu:1}~(ii)-(iii)}  and~{\rm\ref{asu:2}},  given a compact set $K_0\subset \R^n$, the following statements hold:
\begin{itemize}
  \item[\rm{(i)}] There is a global attractor $\A\subset \mathcal{H} \times K_0 \times \R^m$ for the globally defined skew-product semiflow $\pi$ given in~\eqref{eq:sk-pr} which can be written as
\[
\A=\bigcup_{(h,x)\in\mathcal{H}\times K_0} \{(h,x)\}\times A_{(h,x)}
\]
for the sections $A_{(h,x)}=\{ y\in \R^m\mid (h,x,y)\in \A \}$, and the nonautonomous set $\{A_{(h,x)}\}_{(h,x)\in \mathcal{H}\times K_0}$ is the pullback attractor of the skew-product semiflow.
\item[\rm{(ii)}] For each fixed $x\in K_0$, the skew-product semiflow $\pi^x$ given in~\eqref{eq:sk-pr x} has a global attractor $\A^x\subset \mathcal{H}  \times \R^m$ which can be written as
\[
\A^x=\bigcup_{h\in\mathcal{H}} \{h\}\times A^x_h
\]
for the sections $A^x_h=\{ y\in \R^m\mid (h,y)\in \A^x \}$, and the nonautonomous set $\{A^x_h\}_{h\in \mathcal{H}}$ is the pullback attractor of the skew-product semiflow $\pi^x$.
\end{itemize}
   \end{prop}
\begin{proof}
By Assumption~{\rm\ref{asu:2}} (ii) we can take $c$ the positive constant in Definition~\ref{def:ultimately}  and consider the closed ball $B$ in $\R^m$ centered at the origin with radius $c$. Then, it follows that, given any bounded set $D\subset \R^m$, there is a time $T=T(D)$ such that
\[
y(\tau+\tau_0,g,x,\tau_0,D) =y(\tau,g{\cdot}\tau_0,x,D)  \subset B
\]
for every $x\in K_0$, $\tau_0\in\R$ and $\tau\geq T$. Since each $h\in \mathcal{H}$ is the limit of a sequence $\{g{\cdot}\tau_k\}_{k\geq 1}$  as $k\to \infty$, for some  $\{\tau_k\}_{k\geq 1}\subset \R$, by the continuity  of the skew-product semiflow we can deduce that also
\[
y(\tau,h,x,D)  \subset B
\]
for every $x\in K_0$ and $\tau\geq T$. In other words, $y(\tau,\mathcal{H},K_0,D) \subset B$ for every $\tau\geq T$ and therefore, $\mathcal{H}\times K_0\times B$ is an absorbing compact set for the skew-product semiflow~\eqref{eq:sk-pr}, as well as $\mathcal{H}\times  B$ is an absorbing compact set for the semiflow $\pi^x$ for each fixed $x\in K_0$. It is well-known that then, each of these semiflows has a global attractor (see, e.g.,~\cite[Theorem~1.36]{klra}).

To finish, it is well-known (see, e.g.,~\cite[Chapter 3]{klra} or~\cite[Chapter~16]{book:CLR}) that, with a compact base, as it is the case,  if the global attractor exists,  it can be written as expressed in (i) and (ii), respectively, and the fiber sections form the pullback attractor in each case.
The proof is finished.
\end{proof}
Note that, by construction, $(h,x,y)\in \A \Leftrightarrow (h,y)\in \A^x$ for each $(h,x,y)\in \mathcal{H}\times K_0\times \R^m$. Equivalently,  for the sections, $A_{(h,x)}=A_h^x$ for $(h,x)\in \mathcal{H}\times K_0$.

The condition of uniform ultimate boundedness of the solutions of the whole family of equations over $K_0$ requires this property for each equation, but also that the family of attractors $\{\A^x\}_{x\in K_0}$ be bounded.
\begin{prop}\label{prop:unifom ulti bound on K}
Let $K_0\subset \R^n$ be compact. Under Assumptions~{\rm\ref{asu:1}~(ii)-(iii)} and~{\rm \ref{asu:2} (i)}, suppose that for each fixed $x\in K_0$ the solutions of $y'=g(x,y,\tau)$ are uniformly ultimately bounded, with global attractor $\A^x$ for the associated skew-product semiflow $\pi^x$. If there exists an  $r>0$ such that
\[
\bigcup_{x\in K_0} \A^x \subset \mathcal{H}\times \overline B(0,r)\subset \mathcal{H}\times \R^m\,,
\]
then the solutions of the parametric family of equations $y'=g(x,y,\tau)$ for $x\in K_0$ are uniformly ultimately bounded.
\end{prop}
\begin{proof}
First of all, note that the application of Proposition~\ref{prop:existe atractor} to the compact set $\{x\}$ permits to build the global attractor $\A^x$ for $\pi^x$,  for every fixed $x\in K_0$. Arguing by contradiction, assume that for the positive value $r+1$, there is a $d>0$ such that, for every time $T_k=k\geq 1$ there exist $\tau_k\in\R$, $\sigma_k\geq k$, $x_k\in K_0$ and $y_k\in \R^m$ with $|y_k|\leq d$ such that
\begin{equation}\label{eq:r+1}
  |y(\sigma_k+\tau_k,g,x_k,\tau_k,y_k)| = |y(\sigma_k,g{\cdot}\tau_k,x_k,y_k)|> r+1\quad\hbox{for all}\;\,k\geq 1\,.
\end{equation}
Note that, by compactness, we can assume without loss of generality that  $x_k\to x_0 \in K_0$.  Let $c_0=c_0(x_0)>0$ be the constant given in Definition \ref{def:ultimately} and note that we can assume that $d\geq c_0+1$. Associated to $d$,  let $T=T(d,x_0)>0$ be such that
\[
|y(\tau+\tau_0,g,x_0,\tau_0,y_0)|= |y(\tau,g{\cdot}\tau_0,x_0,y_0)|\leq c_0 \;\,\hbox{for all}\; \tau_0\in\R\,,\;\tau\geq T\;\hbox{and}\;\,|y_0|\leq d\,.
\]
Arguing as in the proof of Proposition \ref{prop:existe atractor}, for $D=\overline B(0,d)\subset\R^m$ we deduce that
\[
\big|y(\tau,h,x_0,D)\big|\leq c_0 \quad\hbox{for all}\;\, \tau\geq T\;\hbox{and}\;\,h\in \mathcal{H}\,.
\]
By an argument of continuity, since $x_k\to x_0$, there exists a $k_0$ such that
\[
\big|y(T,h,x_k,D)\big|\leq c_0+1 \leq d \quad\hbox{for all}\;\, h\in \mathcal{H} \;\hbox{and}\;\, k\geq k_0\,,
\]
and this permits to prove by induction, using the cocycle property \eqref{eq:cocycle} that
\begin{equation}\label{eq:B}
  \big|y(jT,h,x_k,D)\big|\leq c_0+1 \leq d \quad\hbox{for all}\;\, h\in \mathcal{H}\,,\; j\geq 1 \;\hbox{and}\;\, k\geq k_0\,.
\end{equation}
Now consider the compact set $K := y([0,T],\mathcal{H},K_0,D)\subset \R^m$ and let us see that
$y(\tau,g{\cdot}\tau_k,x_k,y_k)\in K$ for all $\tau\geq 0$ and $k\geq k_0$. If $\tau=j T +s$ for some $j\in \N$ and $s\in [0,T)$, once more by the cocycle identity, and thanks to \eqref{eq:B},
\[
y(\tau,g{\cdot}\tau_k,x_k,y_k)=y(s,g{\cdot}(\tau_k+jT),x_k,y(j T,g{\cdot}\tau_k,x_k,y_k) )\in K\;\,\, \hbox{for}\;\,\tau\geq 0\,\;k\geq k_0\,.
\]
Then, we can assume without loss of generality that $y(\sigma_k,g{\cdot}\tau_k,x_k,y_k)\to z_0$ and $g{\cdot}(\tau_k+\sigma_k)\to h_0$ as $k\to\infty$.  Let us check that $(h_0,z_0)$ admits an entire bounded orbit for $\pi^{x_0}$. For each $s\geq 0$, by the cocycle relation and continuity,  as $k\to \infty$,
\[
y(s+\sigma_k,g{\cdot}\tau_k,x_k,y_k) = y(s,g{\cdot}(\tau_k+\sigma_k),x_k,y(\sigma_k,g{\cdot}\tau_k,x_k,y_k))\to y(s,h_0,x_0,z_0)\in K.
\]
Consider $I:=\{\tau<0\mid y(\cdot,h_0,x_0,z_0)\;\hbox{is defined on}\;[\tau,0]\}$. Note that $I\neq \emptyset$ and that, given $\tau\in I$, the previous identity makes sense for $k$ big enough and $s\in [\tau,0]$. A standard argument of prolongation of solutions implies that $I=(-\infty,0)$ and hence $y(s,h_0,x_0,z_0)\in K$ for all $s\in \R$. As a consequence $(h_0,z_0)\in\A^{x_0}$ and $\A^{x_0}\subset \mathcal{H}\times \overline B(0,r)$ by hypothesis. Thus, $|z_0|\leq r$, which contradicts~\eqref{eq:r+1}. Therefore, the solutions of the parametric family for $x\in K_0$ are uniformly ultimately bounded, as we wanted to prove.
\end{proof}

\begin{rmks}\label{rmk:atractor uniforme}
(1)~It is important to note that, under Assumptions~\ref{asu:1} and~\ref{asu:2}, we can apply Theorem~5.3 in Artstein~\cite{paper:ZA99}. First of all, we are working in a more regular context for $g$ than the Carath\'{e}odory one in~\cite[Assumption~2.1]{paper:ZA99}. With regard to~\cite[Assumption~5.1]{paper:ZA99}, we can take the open set $G(x):=\R^m$ and the compact set $K(x):=\bigcup_{h\in\mathcal{H}} A_h^x\subset \R^m$ for each $x\in\R^n$ and then, arguing as in the proof of the previous proposition, it is easy to check that if $x_k\to x_0$ in $\R^n$,  $y_k\to y_0$ in $\R^m$, $\tau_k\geq 0$ and $\sigma_k\to\infty$, then $y(\sigma_k+\tau_k,g,x_k,\tau_k,y_k)$ converges to the compact set $K(x_0)$. Besides, it is easy to check that the restriction of the set-valued map $K(\cdot)$ to a compact subset of $\R^n$ has a compact graph, that is, given a compact set $K_0\subset \R^n$, the set $\{(x,y)\mid x\in K_0, \,y\in K(x)\}$ is compact.

(2)~As it has already been remarked in Section~\ref{secpreli}, a pullback attractor in general need not be forward attracting, in the sense that for each $(h,x)\in\mathcal{H}\times K_0$,
\begin{equation*}
\lim_{\tau\to \infty}{\rm dist}\big(y(\tau,h,x,D),A_{(h{\cdot}\tau,x)}\big)=0   \quad \hbox{for each bounded set}\; D\subset  \R^m\,
\end{equation*}
for the skew-product semiflow $\pi$; or in the sense that for each $h\in\mathcal{H}$,
\begin{equation*}
\lim_{\tau\to \infty}{\rm dist}\big(y(\tau,h,x,D),A^x_{h{\cdot}\tau}\big)=0   \quad \hbox{for each bounded set}\; D\subset  \R^m\,
\end{equation*}
for the skew-product semiflow $\pi^x$, for each fixed $x\in K_0$.

This is the case when the pullback attractor is uniformly pullback attracting, for every bounded set $D\subset  \R^m$.
Note that this forwards convergence is for the cocycle mapping and is different from the one in the definition of global attractor.

In the literature there are different approaches to describe some form of forwards convergence of the cocycle mapping, even if the attracting sets lack the desirable property of invariance. We follow the approach in~\cite{klra} by considering the {\em parametrically inflated pullback attractors\/}. This concept was introduced in~\cite{walikl}. Precisely, given $\delta>0$, the nonautonomous set $\{A_{(h,x)}[\delta]\}_{(h,x)\in \mathcal{H}\times K_0}$ with fibers defined for each $(h,x)\in \mathcal{H}\times K_0$ by
\begin{equation}\label{eq:inflated}
  A_{(h,x)}[\delta] := \bigcup_{{\rm d}((\tilde h,\tilde x),(h,x))\leq \delta} A_{(\tilde h,\tilde x)}
\end{equation}
is the so-called {\it parametrically inflated pullback attractor\/}. Here we are considering the distance ${\rm d}((\tilde h,\tilde x),(h,x)):=\sup\{d(h,\tilde h),|x-\tilde x|\}$ in the product space $\mathcal{H}\times K_0$.

Analogously, for each $x\in K_0$ fixed, given $\delta>0$ the parametrically inflated pullback attractor for $\pi^x$ is the nonautonomous set $\{A^x_{h}[\delta]\}_{h\in \mathcal{H}}$ with fibers
\begin{equation}\label{eq:inflated 2}
A^x_{h}[\delta] := \bigcup_{d(\tilde h,h)\leq \delta} A^x_{\tilde h}\,.
\end{equation}
It is convenient to note that the fiber sets $A_{(h,x)}[\delta]$ or $A^x_{h}[\delta]$ are nonempty compact subsets of $\R^m$ (see~\cite[Lemma~3.39]{klra}).
\end{rmks}

The next result on forward attraction for the cocycle mapping will be crucial in the proofs of the main theorems in Section~\ref{secTracking}.
In short, if some additional continuity assumption on the setvalued mappings $x\in K_0\mapsto \A^x$  or $(h,x)\in \mathcal{H}\times K_0\mapsto A_{(h,x)}$ holds, the description is improved, in the sense that the forward attraction sets get smaller.
We give names to these assumptions, for future reference. In both of them, the existence of the attractors is given for granted.
\begin{assu}\label{assu:cont A^x}
The setvalued mapping $x\in K_0\mapsto \A^x\subset \mathcal{H}\times \R^m$ is continuous for the Hausdorff distance $d_H$.
\end{assu}

\begin{assu}\label{assu:cont A_(h,x)}
The setvalued mapping $(h,x)\in \mathcal{H}\times K_0\mapsto A_{(h,x)}\subset \R^m$ is continuous for the Hausdorff distance $d_H$.
\end{assu}

\begin{thm}\label{teor:param inflated pullb att}
Under Assumptions~{\rm \ref{asu:1}~(ii)-(iii)} and~{\rm\ref{asu:2}},  given a compact set $K_0\subset \R^n$, let $\A\subset \mathcal{H} \times K_0 \times \R^m$  be the global attractor  for the skew-product semiflow~\eqref{eq:sk-pr} and let $\A^x\subset \mathcal{H} \times \R^m$  be the global attractor  for the skew-product semiflow~\eqref{eq:sk-pr x} for each fixed $x\in K_0$. Then, fixed a $\delta>0$,  for every bounded set $D\subset \R^m$,
\begin{equation}\label{eq:unif atract}
  \lim_{\tau\to\infty} \bigg( \sup_{h\in \mathcal{H},\, x\in K_0} {\rm dist}\big(y(\tau,h,x,D),A_{(h{\cdot}\tau,x)}[\delta]\big) \bigg) =0\,
\end{equation}
and, if Assumption~{\rm \ref{assu:cont A^x}} holds, then
\begin{equation}\label{eq:unif atract 2}
  \lim_{\tau\to\infty} \bigg( \sup_{h\in \mathcal{H},\, x\in K_0} {\rm dist}\big(y(\tau,h,x,D),A^x_{h{\cdot}\tau}[\delta]\big) \bigg) =0\,.
\end{equation}
If Assumption~{\rm \ref{assu:cont A_(h,x)}} holds, then there is uniform forward attraction, that is, for  every bounded set $D\subset \R^m$,
\begin{equation}\label{eq:unif forw}
  \lim_{\tau\to\infty} \bigg( \sup_{h\in \mathcal{H},\, x\in K_0} {\rm dist}\big(y(\tau,h,x,D),A_{(h{\cdot}\tau,x)}\big) \bigg) =0\,.
\end{equation}
\end{thm}
\begin{proof}
With no confusion in the notation, let us consider on $\mathcal{H}\times\R^n\times \R^m$ the metric given by ${\rm d}((h,x,y),(\tilde h,\tilde x,\tilde y)):=\sup\{d(h,\tilde h),|x-\tilde x|,|y-\tilde y|\}$, and on $\mathcal{H}\times \R^m$ the metric given by ${\rm d}((h,y),(\tilde h,\tilde y)):=\sup\{d(h,\tilde h),|y-\tilde y|\}$. The proof of the first two relations relies on applying the global attraction of $\A$ to the bounded set $\mathcal{H} \times K_0 \times D$. Fixed a $\delta>0$, let us take an $\ep>0$ such that $\ep\leq \delta$. Then, there is a $T=T(D,\ep)>0$ such that for every $ \tau\geq T$,
\[
{\rm dist}\bigg(\pi(\tau,\mathcal{H},K_0,D),\bigcup_{(h,x)\in\mathcal{H}\times K_0} \{(h,x)\}\times A_{(h,x)}\bigg)<\ep \,.
\]
Then, for each $\tau\geq T$, $h\in \mathcal{H}$, $x\in K_0$ and $y_0\in D$ we can find  $(\tilde h,\tilde  x)\in \mathcal{H}\times K_0$ and $\tilde y\in A_{(\tilde h,\tilde x)}$ such that $d(h{\cdot}\tau,\tilde h)<\ep$, $|x-\tilde x|<\ep$ and $|y(\tau,h,x,y_0)-\tilde y|<\ep$. In particular ${\rm d}((h{\cdot}\tau,x),(\tilde h,\tilde x))<\ep\leq \delta$, so that by the definition $\tilde y\in A_{(h{\cdot}\tau,x)}[\delta]$. Therefore,
\[
{\rm dist} \big(y(\tau,h,x,y_0),A_{(h{\cdot}\tau,x)}[\delta]\big)<\ep\,.
\]
From this, by taking superiors,~\eqref{eq:unif atract} follows.

The argument to prove~\eqref{eq:unif atract 2} is similar. Firstly, with Assumption~{\rm \ref{assu:cont A^x}} we note that the setvalued mapping $x\mapsto \A^x$ is in fact uniformly continuous on the compact set $K_0$. Then, for every $\ep>0$, with $\ep\leq \delta$,  we can find a $\tilde \delta>0$ with $\tilde \delta< \ep/2$ such that whenever $x,\tilde x\in K_0$ satisfy $|x-\tilde x|\leq \tilde \delta$, then $d_H(\A^x,\A^{\tilde x})<\ep/2$. Since $\A$ is the global attractor for $\pi$, this time we take a $T=T(D,\tilde\delta)>0$ such that for every $\tau\geq T$,
\[
{\rm dist}\bigg(\pi(\tau,\mathcal{H},K_0,D),\bigcup_{(h,x)\in\mathcal{H}\times K_0} \{(h,x)\}\times A_{(h,x)}\bigg)<\tilde \delta \,.
\]
Now, for each $\tau\geq T$, $h\in \mathcal{H}$, $x\in K_0$ and $y_0\in D$ we can find  $(\tilde h,\tilde  x)\in \mathcal{H}\times K_0$ and $\tilde y\in A_{(\tilde h,\tilde x)}$ such that $d(h{\cdot}\tau,\tilde h)<\tilde \delta$, $|x-\tilde x|<\tilde \delta$ and $|y(\tau,h,x,y_0)-\tilde y|<\tilde \delta$. In particular $|x-\tilde x|<\tilde \delta$ implies that $d_H(\A^x,\A^{\tilde x})<\ep/2$. Since $(\tilde h,\tilde y)\in \A^{\tilde x}$, we can take $(h^*,y^*)\in\A^x$ such that $d(\tilde h,h^*)<\ep/2$ and $|\tilde y-y^*|<\ep/2$. Then, $d(h{\cdot}\tau,h^*)\leq d(h{\cdot}\tau,\tilde h) +d(\tilde h,h^*)<\tilde\delta+\ep/2< \ep\leq \delta$. Since in particular $y^*\in A^x_{h^*}$, by the definition, $y^*\in A^x_{h{\cdot}\tau}[\delta]$ and $|y(\tau,h,x,y_0)-y^*|\leq |y(\tau,h,x,y_0)-\tilde y|+|\tilde y-y^*|<\tilde\delta+\ep/2< \ep$. That is, ${\rm dist} \big(y(\tau,h,x,y_0),A^x_{h{\cdot}\tau}[\delta]\big)<\ep$.
From here, the proof of~\eqref{eq:unif atract 2} is finished as before.

Finally, with the continuity in Assumption~\ref{assu:cont A_(h,x)}, it is proved in~\cite[Theorem~4.2]{chks} that the pullback attractor
 $\{A_{(h,x)}\}_{(h,x)\in \mathcal{H}\times K_0}$ is a uniform pullback attractor, that is, for every bounded set $D\subset \R^m$,
\begin{equation*}
  \lim_{\tau\to\infty} \bigg( \sup_{h\in \mathcal{H},\, x\in K_0} {\rm dist}\big(y(\tau,h{\cdot}(-\tau),x,D),A_{(h,x)}\big) \bigg) =0\,.
\end{equation*}
As it is easy to check, this is equivalent to the uniform forward attraction, as expressed in~\eqref{eq:unif forw}. The proof is finished.
\end{proof}
Although the assertion that continuity of the attractors is equivalent to equi-attraction  can be found in different references in different contexts, to finish this section  we include a short proof of this result (one implication is adapted from Li and Kloeden~\cite[Theorem~2.9]{paper:LiK}). Let us first give the definition of equi-attraction.
\begin{defn}
 Fixed a compact set $K_0\subset \R^n$, the family of attractors $\{\A^x\mid x\in K_0\}$ is said to be {\em equi-attracting\/} if given $\ep>0$ and a bounded set $D\subset \R^m$ there exists $T=T(\ep,D)$ independent of $x\in K_0$ such that ${\rm dist}\big(\pi^x(\tau,\mathcal{H},D),\A^x\big)\leq \ep$ for every $\tau\geq T$ and $x\in K_0$.
\end{defn}
\begin{thm}\label{teor:equivalence}
Under Assumptions~{\rm \ref{asu:1}~(ii)-(iii)} and~{\rm\ref{asu:2}}, given a compact set $K_0\subset \R^n$, the following statements are equivalent:
\begin{itemize}
  \item[\rm{(1)}] The map $x\in K_0\mapsto \A^x\subset \mathcal{H}\times \R^m$ is continuous for the Hausdorff distance.
  \item[\rm{(2)}] The family of attractors $\{\A^x\mid x\in K_0\}$ are equi-attracting.
\end{itemize}
\end{thm}
\begin{proof}
To prove $(1)\Rightarrow (2)$, we argue as in the proof of~\eqref{eq:unif atract 2}, omitting $\delta$.  Briefly, given an $\ep>0$ there is a $\tilde\delta>0$ associated to the uniform continuity of the map in (1), and given a bounded set $D\subset \R^m$, there exists a $T=T(D,\tilde\delta)>0$ such that  for each $\tau\geq T$, $h\in \mathcal{H}$, $x\in K_0$ and $y_0\in D$, a pair $(h^*,y^*)\in\A^x$ is found so that ${\rm d}(\pi^x(\tau,h,y_0),(h^*,y^*))=\sup\{d(h{\cdot}\tau,h^*), |y(\tau,h,x,y_0)-y^*| \}<  \ep$, and thus ${\rm dist} (\pi^x(\tau,h,y_0),\A^x)<\ep$. From this it follows that ${\rm dist}\big(\pi^x(\tau,\mathcal{H},D),\A^x\big)\leq \ep$ for every $\tau\geq T$ and $x\in K_0$.

Conversely, assume that $\{\A^x\mid x\in K_0\}$ are equi-attracting and take a compact set $D\subset \R^m$ such that $\bigcup_{h\in\mathcal{H},x\in K_0} A^x_h\subset D$. Given $\ep>0$, the equi-attraction assumption provides a $T>0$ such that ${\rm dist}\big(\pi^x(\tau,\mathcal{H},D),\A^x\big)< \ep/2$ for every $\tau\geq T$ and $x\in K_0$. Now, by the joint continuity of the cocycle mapping in~\eqref{eq:sk-pr} and the compactness of the sets involved, there exists a $\delta>0$ such that, provided that $x, \tilde x\in K_0$ satisfy $|x-\tilde x|<\delta$,
$|y(T,h,x,z)-y(T,h,\tilde x,z)|< \ep/2$ for every  $h\in\mathcal{H}$ and $z\in D$.
From this it follows that ${\rm dist} (\pi^x(T,\A^x),\pi^{\tilde x}(T,\A^x))<\ep/2$ whenever  $|x-\tilde x|<\delta$. Then, by the $\pi^x$-invariance of $\A^x$ and the triangle inequality, for $x, \tilde x\in K_0$ with $|x-\tilde x|<\delta$,
\begin{align*}
  {\rm dist} (\A^x,\A^{\tilde x})= &  \,{\rm dist} (\pi^x(T,\A^x),\A^{\tilde x})\\
  \leq  & \,{\rm dist}(\pi^x(T,\A^x), \pi^{\tilde x}(T,\A^x)) + {\rm dist} ( \pi^{\tilde x}(T,\A^x), \A^{\tilde x})< \ep
\end{align*}
applying the equi-attraction property to the subset $\A^x\subset \mathcal{H}\times D$. By the definition of the Hausdorff metric, we have actually proved the uniform continuity on $K_0$. The proof is finished.
\end{proof}
\section{A dynamical approach to the fast motion and tracking of nonautonomous attractors}\label{secTracking}
This section and Section~\ref{secOmegalim} contain the main results of the paper on the dynamical behaviour of the fast variables $y_\ep(\tau)$. In the present section, we offer a first version of our results in Theorems~\ref{teor:main} and~\ref{teor:valor final} under Assumptions~\ref{asu:1} and~\ref{asu:2}, whose proofs are the core of the paper. Next, some additional conditions are added, leading to more precise dynamical descriptions applicable in a wide range of situations. We keep the notation introduced in the previous section.

Theorem~\ref{teor:main} is the first in a series of results which offer nontrivial extensions of Theorem~5.3 in Artstein~\cite{paper:ZA99} in what refers to the understanding of the behaviour of the fast motion. The cited result in~\cite{paper:ZA99} basically proves uniform convergence for the slow variables and statistical convergence for the fast variables. Let us briefly explain the meaning of this statistical convergence. Let $B$ be a compact ball in $\R^m$ such that $y_\ep(t)\in B$ for $t\in [0,t_0]$ and $0<\ep\leq \ep_0$ (see Proposition~6.6 in~\cite{paper:ZA99}). For a sequence $\ep_k\downarrow 0$ ($\ep_k\leq \ep_0$) such that $x_{\ep_k}(t)$ converge uniformly to $x_0(t)$ on $[0,t_0]$, associated to each solution $y_{\ep_k}(t)$,  $t\in [0,t_0]$ a measure is built $\widehat y_{\ep_k}\in\mathcal{\M}=\mathcal{\M}([0,t_0]\times B)$, the set of normalized Borel measures on $[0,t_0]\times B$,  by means of
\[
\int_{[0,t_0]\times B} \varphi\,d\widehat y_{\ep_k} = \frac{1}{t_0}\int_0^{t_0} \varphi(t,y_{\ep_k}(t))\,dt\quad\hbox{for each}\; \varphi\in C([0,t_0]\times B)\,.
\]
The statistical convergence of $y_{\ep_k}$  means that  $\lim_{k\to\infty} {\rm dist}\big(\widehat y_{\ep_k},\mathcal{\M}_0\big)=0$ for the set
\[
\mathcal{\M}_0:=\bigg\{\nu\in \mathcal{\M}\left| \; \nu_t\in M\big(\A^{x_0(t)}_{\Om(g)}\big) \hbox{ for the desintegration } \{\nu_t\}_{t\in [0,t_0]} \hbox{ of }\nu  \right.\bigg\},
\]
that is, the sequence of measures $\widehat y_{\ep_k}$ converges to the set $\mathcal{\M}_0$ in the weak topology. The set $M\big(\A^x_{\Om(g)}\big)$ denotes the collection of marginals (or projections on $\R^m$) of $\pi^x$-invariant measures supported on the attractor $\A^x$ restricted to the omega-limit set of $g$, that is, on the set $\A^x_{\Om(g)}:=\A^x\cap (\Om(g)\times \R^m)$.

The novelty in our work is that we get to explain the asymptotical behaviour of the fast variables from a purely dynamical point of view. The results can be read as a process of tracking of nonautonomous attractors whenever $\ep$ is small enough and we look at the location of the fast motion $y_{\ep}(\tau)$ as $\tau$ moves on an interval $[T,t_0/\ep]$.

The following lemma on the separation of solutions of close differential equations is included for completeness. We omit the proof since it follows from a direct application of the standard results on  $\ep$-approximate solutions of a differential equation (see, e.g., Theorem~1.2.1 in~\cite{book:CL}).
\begin{lem}\label{lem:comparacion}
Let $E$ be an open set of $\R^{m+1}$ and let  $g_1,g_2:E\to \R^m$ be continuous maps and consider the systems of ODEs $y'=g_1(y,\tau)$, $y'=g_2(y,\tau)$. Assume further that $g_1$ is Lipschitz in $E$ with respect to $y$ with constant $L>0$. If $y_1,\, y_2:I\to \R^m$ are solutions of $y'=g_1(y,\tau)$, $y'=g_2(y,\tau)$, respectively, with  $y_1(\tau_0)=y_2(\tau_0)$ for a $\tau_0\in I$, and
\[
|g_1(y_2(\tau),\tau)-g_2(y_2(\tau),\tau)|\leq \sigma\quad \hbox{for every }\;\tau\in I,
\]
 then
\[
|y_1(\tau)-y_2(\tau)|\leq \frac{\sigma}{L}\,\big(e^{L|\tau-\tau_0|}-1\big)\quad \hbox{for every }\;\tau\in I.
\]
\end{lem}
We now state the first result. Recall the definition of the  parametrically inflated pullback attractor given in~\eqref{eq:inflated}.
\begin{thm}\label{teor:main}
Under Assumptions~{\rm \ref{asu:1}} and~{\rm\ref{asu:2}}, for each $\ep>0$  let $(x_\ep(t),y_\ep(t))$ be the solution of~\eqref{eq:slowfastent}-\eqref{eq:initial values}. Then, given $t_0>0$, for $\ep$ small enough, $(x_\ep(t),y_\ep(t))$ can be extended to $t\in [0,t_0]$. Furthermore:
\begin{itemize}
\item[\rm{(i)}]  Every sequence $\ep_j\to 0$ has a subsequence, say $\ep_k\downarrow 0$, such that $x_{\ep_k}(t)$ converges uniformly for $t\in [0,t_0]$ to a solution of the differential inclusion
\begin{equation}\label{eq:diff incl}
\dot{x}  \in \left\{ \int_{\R^m} f(x,y)\,d\mu(y) \left| \;\mu\in M\big(\A^x_{\Om(g)}\big)\right.\right\},\quad x(0)=x_0\,.
\end{equation}
Let the limit be $x_0(t)$, for $t\in [0,t_0]$.
\item[\rm{(ii)}] For the fast variables in the fast timescale $y_{\ep_k}(\tau)$, the behaviour is as follows. For a fixed small $r>0$, let $K_0$ be the tubular set
\begin{equation}\label{eq:K0}
K_0:=\big\{x_0(t)+x\mid t\in [0,t_0],\;x\in\R^n \;{\rm with}\;|x|\leq r\big\}\subset \R^n
\end{equation}
which is compact and let $\A\subset \mathcal{H} \times K_0 \times \R^m$ be the global attractor given in Proposition~{\rm\ref{prop:existe atractor}} for the  skew-product semiflow~\eqref{eq:sk-pr}.  Then, given $\delta >0$, there exist $T=T(\delta,y_0,\A)>0$ and $k_0=k_0(\delta,T)>0$ with $2T<t_0/\ep_{k_0}$ such that for every $k\geq k_0$,
\begin{equation}\label{eq:teorema}
 {\rm dist}\big(y_{\ep_k}(\tau),A_{(g{\cdot}\tau,x_0(\ep_k\tau))}[\delta]\big)\leq  \delta \quad \hbox{for all }\; \tau\in [T,t_0/\ep_k]\,.
\end{equation}
\end{itemize}
\end{thm}
\begin{proof}
First of all, as it has been noted in Remark~\ref{rmk:atractor uniforme}~(1), we can apply Theorem~5.3 in~\cite{paper:ZA99}.
The statement in (i) above on the behaviour of the slow variables reproduces what this theorem says, except for actualizing the set of measures in the differential inclusion.  Going into detail in Artstein's proof, it can be checked that in fact the measures involved are the projections of the $\pi^x$-invariant measures supported on $\A^x_{\Om(g)}$.

Let us now prove the statement in (ii) for the fast variables.   For convenience, from now on for each $\ep>0$ we write the slow variables with respect to the slow timescale $x_\ep(t)$, $t\in [0,t_0]$, but we write the fast variables with respect to the fast timescale $y_\ep(\tau)$, $\tau\in [0,t_0/\ep]$.

Given $0<\delta<1$, first of all let us determine the appropriate value of $T$ as the attraction time for a precise bounded set according to Theorem~\ref{teor:param inflated pullb att}. More precisely, take  $\rho_0>0$ big enough so that
\begin{equation}\label{eq:rho0}
  \bigcup_{(h,x)\in \mathcal{H}\times K_0} A_{(h,x)} \subset \overline B(0,\rho_0)\subset \R^m
\end{equation}
and consider the compact set  $D:=\overline B(0,\rho)\subset \R^m$ for a fixed $\rho\geq \max(|y_0|,\rho_0+1)$. Note that $D$ depends on the attractor $\A$ and the initial value $y_0$. Associated to this set $D$ and to $\delta>0$, it follows from~\eqref{eq:unif atract} in Theorem~\ref{teor:param inflated pullb att} that there is a $T=T(D,\delta)$ such that for every  $(h,x)\in\mathcal{H}\times K_0$,
\begin{equation*}
   {\rm dist}\big(y(\tau,h,x,D),A_{(h{\cdot}\tau,x)}[\delta]\big) < \frac{\delta}{2}
\end{equation*}
for every $\tau\geq T$. For our purposes we will be using that, since by construction  $x_0(\ep_k s)\in K_0$ for every $k\geq 1$ and every $s\in [0,t_0/\ep_k]$, in particular we have  that
\begin{equation}\label{eq:dist tau mayor T}
   {\rm dist}\big(y(\tau,g{\cdot}\gamma,{x_0(\ep_k s)},D),A_{(g{\cdot}(\gamma+\tau),x_0(\ep_k s))}[\delta]\big) < \frac{\delta}{2}
\end{equation}
 for every $\gamma\in \R$, $k\geq 1$ and every $s,\tau\in [0,t_0/\ep_k]$, provided that $\tau\geq T$.
\par\smallskip
Secondly, let us determine the value of $k_0$. This is rather technical. By the continuity of the skew-product semiflow and the compactness of the sets involved, the set
\begin{equation}\label{eq:K}
  K:=y\big([0,2T]\times \mathcal{H}\times K_0\times D \big) \subset \R^m
\end{equation}
is compact. It is a common procedure in nonautonomous dynamics (e.g., see~\cite{shyi}) to write the family of equations $y'=h(x,y,\tau)$, $h\in\mathcal{H}$, $ x\in K_0$ as
\begin{equation}\label{eq:G def}
  y'=G(x,y,h{\cdot}\tau)\,,\quad h\in\mathcal{H}\,,\;  x\in K_0
\end{equation}
for a continuous map $G:K_0\times \R^m\times \mathcal{H} \to \R^m$, defined by $G(x,y,h)=h(x,y,0)$, which extends $g$ in the sense that $G(x,y,g{\cdot}\tau)=g(x,y,\tau)$ for all $x\in K_0$,  $y\in\R^m$ and $\tau\in\R$. In order to apply Lemma~\ref{lem:comparacion}, let us fix a relatively compact open set $U$ of $\R^m$ such that $K\subset U\subset \overline U$. It is easy to check that the family of equations over the hull inherits the Lipschitz character in $y$ expressed in Assumption~\ref{asu:1} (iii)  in the following sense: given the compact set $K_0$ of $\R^n$ and the compact set $\overline U$ of $\R^m$, there exists $L=L(K_0,\overline U)>0$ such that for every $h\in\mathcal{H}$,
\[
      |h(x,y_1,\tau)-h(x,y_2,\tau)|\leq L\,|y_1-y_2|
\]
for every $y_1,y_2\in \overline U$, $x\in K_0$ and $\tau\in \R$; or, equivalently,
\[
      |G(x,y_1,h{\cdot}\tau)-G(x,y_2,h{\cdot}\tau)|\leq L\,|y_1-y_2|
\]
for every $h\in\mathcal{H}$, $y_1,y_2\in \overline U$, $x\in K_0$ and $\tau\in \R$. Have in mind that $K\subset U\subset \overline U$.

Given the values of $\delta$ and $T>0$, and the Lipschitz constant $L>0$ associated to $K_0$ and $\overline U$, let us consider $\sigma>0$ small enough
 so that
\begin{equation}\label{eq:sigma}
\des\frac{\sigma}{L}\,\big(e^{L2T}-1\big)\leq \des\frac{\delta}{2}\,.
\end{equation}

Then,  associated to $\sigma>0$ we can find an integer $k_0$ with  $2T<t_0/\ep_{k_0}$ and so that for every $k\geq k_0$, $x_{\ep_k}(t)\in K_0$ for every $t\in [0,t_0]$ and
\begin{equation}\label{eq:G}
  \big|G(x_{\ep_k}(\ep_k\tau),y,h)-G(x_0(\ep_k s),y,h)\big|\leq \sigma
\end{equation}
for every $\tau,s\in [0,t_0/\ep_k]$ with $|\tau-s|\leq 2T$,  $y\in K$ and $h\in\mathcal{H}$. To see it, write
\begin{align*}
  \big|G(x_{\ep_k}(\ep_k\tau),y,h)-G(x_0(\ep_k s),y,h)\big|\leq & \big|G(x_{\ep_k}(\ep_k\tau),y,h)-G(x_0(\ep_k \tau),y,h)\big| \\
   &  +\big|G(x_{0}(\ep_k\tau),y,h)-G(x_0(\ep_k s),y,h)\big|\,,
\end{align*}
recall the choice of $K_0$ (see~\eqref{eq:K0}) and apply that  $G$ is uniformly continuous on the compact set  $K_0\times K\times \mathcal{H}$, $x_{\ep_k}(t)\to x_0(t)$ uniformly for $t\in [0,t_0]$, $x_0(t)$ is uniformly continuous on $[0,t_0]$, $|\ep_k\tau-\ep_k s|\leq 2T\ep_k$ and $\ep_k\downarrow 0$.

\par\smallskip
Once we have determined the values of $T$ and $k_0$, in order to prove~\eqref{eq:teorema} we follow an iterative process, by considering successive time intervals of length $T$, that is, $[T,2T]$, $[2T,3T]$, and so on,  until the whole interval for the fast timescale $[T,t_0/\ep_k]$ has been swept for each $k\geq k_0$. Note that the number of iterations is always finite and depends on $k$, although we will omit this fact in the notation. This iterative process allows us to check that the fast dynamics $y_{\ep_k}(\tau)$ remains close to the compact set $K$ defined in~\eqref{eq:K} for $\tau\in [T,t_0/\ep_k]$. At each iteration ($i\geq 0$) the trick relies on relation~\eqref{eq:translated} for the solutions. On the one hand, we select appropriate values of $y_{0}^{ik}\in D$ for $k\geq k_0$, and for each fixed $s\in [(i+1)T,(i+2)T]$ we prove that the solution $y(\tau,g,x_0(\ep_k s),iT, y_{0}^{ik})$  is in $K$ and remains close to $y_{\ep_k}(\tau)$ for all $\tau\in[iT,(i+2)T]$. Thus,  in particular, we can approximate  $y_{\ep_k}(s)$  by the value of $\left.y(\tau,g,x_0(\ep_k s),iT, y_{0}^{ik})\right|_{\tau=s}=\left.y(\tau-iT,g{\cdot}(iT),x_0(\ep_k s), y_{0}^{ik})\right|_{\tau=s}$, using~\eqref{eq:translated}. On the other hand, we know that for times $\tau-iT\geq T$, that is, if $\tau\geq (i+1)T$, this solution is within a short distance of the corresponding section of the parametrically inflated pullback attractor, as expressed in~\eqref{eq:dist tau mayor T}.

Let us explain in detail the first and second steps of this iterative process.
\par\smallskip

{\it First step\/} ($i=0$): Let us start by proving~\eqref{eq:teorema} for $\tau\in [T,2T]$ for $k\geq k_0$. For each $k\geq k_0$,    let $y_0^{0k}=y_0\in D$, the one given in~\eqref{eq:initial values}.

Let us fix $s\in [T,2T]$. On the interval $[0,2T]$ of length $2T$ we can apply Lemma~\ref{lem:comparacion}. More precisely, for each $k\geq k_0$ we have that:
\begin{itemize}
  \item $y_{\ep_k}(\tau)$ is a solution of $y'=g(x_{\ep_k}(\ep_k\tau),y,\tau)$ for $\tau\in [0,2T]$;
  \item the map $g_1(y,\tau):=g(x_{\ep_k}(\ep_k\tau),y,\tau)$ is continuous on $U \times [0,2T]$ and Lipschitz with respect to $y$ with constant $L>0$ determined above;
  \item $y_2(\tau):=y(\tau,g,x_0(\ep_k s),y_0)$ is a solution of $y'=g(x_0(\ep_k s),y,\tau)$ for $\tau\in [0,2T]$;
  \item $y_{\ep_k}(0)=y_2(0)=y_0\in D$;
  \item $y_2(\tau)\in K$  for every $\tau\in [0,2T]$ by construction (see~\eqref{eq:K});
  \item $\big|g(x_{\ep_k}(\ep_k\tau),y_2(\tau),\tau)-g(x_0(\ep_k s),y_2(\tau),\tau)\big|\leq \sigma$ for every $\tau\in [0,2T]$, because of~\eqref{eq:G}.
\end{itemize}
Therefore, for $k\geq k_0$ and for every $\tau\in [0,2T]$,
\begin{equation*}
  |y_{\ep_k}(\tau)-y_2(\tau)|\leq \des\frac{\sigma}{L}\,\big(e^{L(\tau-0)}-1\big)\leq \des\frac{\sigma}{L}\,\big(e^{L2T}-1\big)\leq \des\frac{\delta}{2}\,,
\end{equation*}
by the previous choice of $\sigma$ in~\eqref{eq:sigma}.  Then, by evaluating at $\tau=s$ we can conclude that for every $k\geq k_0$ and $s\in [T,2T]$,
\begin{equation}\label{eq:previa}
\big|y_{\ep_k}(s)-\left.y(\tau,g,x_0(\ep_k s),y_0)\right|_{\tau=s}\big|\leq  \des\frac{\delta}{2}\,.
\end{equation}
Now, note that for $\tau\geq T$, since $y_0\in D$,  ${\rm dist}\big(y(\tau,g,{x_0(\ep_k s)},y_0),A_{(g{\cdot}\tau,x_0(\ep_k s))}[\delta]\big) \leq {\rm dist}\big(y(\tau,g,{x_0(\ep_k s)},D),A_{(g{\cdot}\tau,x_0(\ep_k s))}[\delta]\big) <\delta/2$ by~\eqref{eq:dist tau mayor T}. Once more evaluating at $\tau=s$, this, combined with~\eqref{eq:previa}, allows us to deduce that for every $k\geq k_0$, ${\rm dist}\big(y_{\ep_k}(s),A_{(g{\cdot}s,x_0(\ep_k s))}[\delta]\big)\leq  \delta$ for every $s\in [T,2T]$, i.e.,~\eqref{eq:teorema} holds on the interval $[T,2T]$.

Note that in particular, ${\rm dist}\big(y_{\ep_k}(T),A_{(g{\cdot}T,x_0(\ep_k T))}[\delta]\big)\leq  \delta$. Therefore, by~\eqref{eq:rho0} also ${\rm dist}\big(y_{\ep_k}(T), \overline B(0,\rho_0)\big)\leq  \delta<1$ and by the choice of $D$ it follows that $y_{\ep_k}(T)\in  D$. This is our starting point for the second step.

\par\smallskip
{\it Second step\/} ($i=1$): Let us explain how we fall into an iterative process, by taking a new step. To prove~\eqref{eq:teorema} on the interval $[2T,3T]$, assume for simplicity of writing that $3T<t_0/\ep_k$ for $k\geq k_0$  and fix $s\in [2T,3T]$ (otherwise, we would take $s\in [2T,t_0/\ep_k]$). This time, we apply Lemma~\ref{lem:comparacion} on the interval $[T,3T]$ by taking as initial value at time $T$ for each $k\geq k_0$, $y_{0}^{1k}:=y_{\ep_k}(T)\in D$ (as mentioned in the previous paragraph) and considering the following facts for each $k\geq k_0$:
\begin{itemize}
  \item $y_{\ep_k}(\tau)$ is a solution of $y'=g(x_{\ep_k}(\ep_k\tau),y,\tau)$ for $\tau\in [T,3T]$;
  \item the map $g_1(y,\tau):=g(x_{\ep_k}(\ep_k\tau),y,\tau)$ is continuous on $U \times [T,3T]$ and Lipschitz with respect to $y$ with constant $L>0$ determined above;
  \item $y_2(\tau):=y(\tau,g,x_0(\ep_k s),T,y_0^{1k})$ is a solution of $y'=g(x_0(\ep_k s),y,\tau)$ for $\tau\in [T,3T]$;
  \item $y_{\ep_k}(T)=y_2(T)= y_0^{1k}\in  D$;
  \item $y_2(\tau)\in K$, given in~\eqref{eq:K}, for every $\tau\in [T,3T]$, since by relation~\eqref{eq:translated},    $y_2(\tau)=y(\tau-T,g{\cdot}T,x_0(\ep_k s),y_{0}^{1k})$,  $\tau-T\in [0,2T]$  and $y_{0}^{1k}\in D$;
  \item $\big|g(x_{\ep_k}(\ep_k\tau),y_2(\tau),\tau)-g(x_0(\ep_k s),y_2(\tau),\tau)\big|\leq \sigma$ for every $\tau\in [T,3T]$, because of~\eqref{eq:G}.
\end{itemize}
Therefore, Lemma~\ref{lem:comparacion} says that for every $\tau\in[T,3T]$ and for every $k\geq k_0$,
\[
|y_{\ep_k}(\tau)-y_2(\tau)|\leq \des\frac{\sigma}{L}\,\big(e^{L(\tau-T)}-1\big)\leq \des\frac{\sigma}{L}\,\big(e^{L2T}-1\big)\leq \des\frac{\delta}{2}\,.
\]
By evaluating at $\tau=s$ we can conclude that for every $k\geq k_0$ and $s\in [2T,3T]$,
\begin{equation}\label{eq:previa2}
\big|y_{\ep_k}(s)-\left.y(\tau-T,g{\cdot}T,x_0(\ep_k s),y_{0}^{1k})\right|_{\tau=s}\big|\leq  \des\frac{\delta}{2}\,.
\end{equation}
Now, note that for $\tau\geq 2T$, $\tau-T\geq T$ and since $y_{0}^{1k}\in D$, we can deduce from~\eqref{eq:dist tau mayor T} that
\[
{\rm dist}\big(y(\tau-T,g{\cdot}T,{x_0(\ep_k s)},y_{0}^{1k}),A_{(g{\cdot}\tau,x_0(\ep_k s))}[\delta]\big) <\delta/2\,.
\]
Again evaluating at $\tau=s$, by the triangle  inequality, using~\eqref{eq:previa2}, we conclude that for every $k\geq k_0$, ${\rm dist}\big(y_{\ep_k}(s),A_{(g{\cdot}s,x_0(\ep_k s))}[\delta]\big)\leq  \delta$ for every $s\in [2T,3T]$, i.e.,~\eqref{eq:teorema} holds on the interval $[2T,3T]$.

Looking ahead to the next step,  since ${\rm dist}\big(y_{\ep_k}(2T),A_{(g{\cdot}2T,x_0(\ep_k 2T))}[\delta]\big)\leq  \delta$, by construction it is deduced that $y_{\ep_k}(2T)\in  D$.

By iterating this procedure a finite number of steps, depending on $k$, the proof of (ii) is finished.
\end{proof}
\begin{rmks}\label{rmk:growth bound}
(1) As we mentioned in Section \ref{secSlowfast}, one can often disregard the growth bound in Assumption~\ref{asu:1}~(i).  In fact, whenever the solutions $(x_\ep(t),y_\ep(t))$ exist and   $x_\ep(t)$ are  uniformly bounded on an interval $[0,t_0]$ for $0<\ep\leq \ep_0$, also $y_\ep(t)$  turn out to be uniformly bounded on  $[0,t_0]$ for $0<\ep\leq \ep_0$ (this is proved in~\cite[Proposition 6.6]{paper:ZA99}, once Assumption~\ref{asu:2} is also imposed). Then, the map $f(x,y)$ in~\eqref{eq:slowfastent} can be substituted by a map $f(x,y)\,\varphi(x,y)$ which satisfies Assumption~\ref{asu:1}~(i), preserving the solutions. It suffices to take a map $\varphi\in C^1(\R^n\times \R^m,\R)$, $0\leq \varphi\leq 1$, which takes the value $1$ if $|x|+|y|\leq 2c$ and takes the value $0$ if $|x|+|y|\geq 2c+1$, for a positive $c>0$ such that $|x_\ep(t)|$, $|y_\ep(t)|\leq c$ for every $t\in [0,t_0]$ and $0<\ep\leq \ep_0$.

(2) If the set $M\big(\A^x_{\Om(g)}\big)=\{\mu^x\}$, that is,  when there is a unique invariant measure for $\pi^x$  concentrated on $\A^x_{\Om(g)}$, and assuming uniqueness of solutions for~\eqref{eq:integral}, then $x_{\ep}(t)$ converges uniformly on $[0,t_0]$ as $\ep\to 0$ to the unique solution of the differential equation
\begin{equation}\label{eq:integral}
\dot{x}  = \int_{\R^m} f(x,y)\,d\mu^x(y)\, ,\quad x(0)=x_0\,.
\end{equation}
This can be understood as a version of the theorem of averaging in the terms described by Artstein~\cite{paper:ZA07} and Sanders and Verhulst~\cite{book:SaVe}.
\end{rmks}
When particular attention is paid to the fast motion at the final time on the fast timescale, or equivalently, to $y_{\ep_k}(t_0)$ on the slow timescale, we offer this result.
\begin{thm}\label{teor:valor final}
Under Assumptions~{\rm \ref{asu:1}} and~{\rm\ref{asu:2}}, assume that $x_{\ep_k}(t)$ converges  to $x_0(t)$ uniformly on $[0,t_0]$ and let the compact set $K_0$ and the global attractor $\A$ be the ones defined in Theorem~{\rm \ref{teor:main}}. For each integer $k\geq 1$, let
\begin{equation*}
 \delta_k:=\inf\big\{ \delta>0 \mid  {\rm dist}\big(y_{\ep_k}(t_0/\ep_k),A_{(g{\cdot}(t_0/\ep_k),x_0(t_0))}[\delta]\big)\leq  \delta  \big\}\,.
\end{equation*}
 Then:
 \begin{itemize}
   \item[\rm{(i)}] $\des\lim_{k\to \infty} \delta_k =0$;
   \item[\rm{(ii)}] $\des\lim_{k\to \infty}  {\rm dist}\big(y_{\ep_k}(t_0/\ep_k),A_{(g{\cdot}(t_0/\ep_k),x_0(t_0))}[\delta_k]\big) =0$;
 \item[\rm{(iii)}] if $g{\cdot}(t_0/\ep_k)\to h_0$ in $(\mathcal{H},d)$ as $k\to\infty$, then $h_0\in \Om(g)$ and
 \begin{equation*}
   \lim_{k\to \infty}  {\rm dist}\big(y_{\ep_k}(t_0/\ep_k),A_{(h_0,x_0(t_0))}\big) =0\,.
 \end{equation*}
 \end{itemize}
\end{thm}
\begin{proof}
Item (i) follows immediately from~\eqref{eq:teorema} in Theorem~\ref{teor:main}~(ii) and the definition of $\delta_k$. Item (i) together with the definition of  $\delta_k$ imply (ii). For (iii), note, first of all, that since $\mathcal{H}$ is compact, there always exists a convergent subsequence of $\{g{\cdot}(t_0/\ep_k)\}$.  For simplicity,  assume that $g{\cdot}(t_0/\ep_k)\to h_0$  as $k\to\infty$.
By the triangle inequality, ${\rm dist}\big(y_{\ep_k}(t_0/\ep_k),A_{(h_0,x_0(t_0))}\big) \leq {\rm dist}\big(y_{\ep_k}(t_0/\ep_k),A_{(g{\cdot}(t_0/\ep_k),x_0(t_0))}[\delta_k]\big) +
{\rm dist}\big(A_{(g{\cdot}(t_0/\ep_k),x_0(t_0))}[\delta_k],A_{(h_0,x_0(t_0))}\big) $. Since the first term in the sum goes to $0$ by (ii), it suffices to see that the second term converges to $0$ too. Argue by contradiction and assume that there exists a $\rho>0$ and there exists a subsequence  $\{k_i\}$ of $\{k\}$ which we still denote by $\{k\}$ such that  ${\rm dist}\big(A_{(g{\cdot}(t_0/\ep_k),x_0(t_0))}[\delta_k],A_{(h_0,x_0(t_0))}\big)\geq \rho$ for all $k$. This means that for each $k$ there is $y_k\in A_{(g{\cdot}(t_0/\ep_k),x_0(t_0))}[\delta_k]$ with ${\rm dist}\big(y_k,A_{(h_0,x_0(t_0))}\big)\geq \rho/2$. The fact that $y_k\in A_{(g{\cdot}(t_0/\ep_k),x_0(t_0))}[\delta_k]$ implies that for each $k$ there exists a pair $(h_k,x_k)\in \mathcal{H}\times K_0$ with $d(g{\cdot}(t_0/\ep_k),h_k)\leq \delta_k$ and $|x_0(t_0)-x_k|\leq \delta_k$ such that $y_k\in A_{(h_k,x_k)}$. Then, by (i),   $(h_k,x_k)\to (h_0,x_0(t_0))$ as $k\to\infty$. By the upper semi-continuity of the setvalued map $(h,x)\to A_{(h,x)}$ (see, e.g.,~\cite[Theorem~3.34]{klra}) it follows that ${\rm dist}\big(A_{(h_k,x_k)},A_{(h_0,x_0(t_0))}\big)\to 0$ as $k\to \infty$. However this contradicts that  $y_k\in A_{(h_k,x_k)}$ and ${\rm dist}\big(y_k,A_{(h_0,x_0(t_0))}\big)\geq \rho/2>0$ for all $k$. The proof is finished.
\end{proof}
Once we have written the proof of Theorem~\ref{teor:main}, we get the next theorem almost for free. Requiring Assumption~\ref{assu:cont A^x} on the continuity of the map $x\mapsto \A^x$, which is equivalent to the  equi-attraction of the family  $\{\A^x\mid x\in K_0\}$ (see Theorem~\ref{teor:equivalence}), we get a more accurate description of the sets in $\R^m$ which the fast variables approach, in terms of the parametrically inflated pullback attractors of  $\pi^x$ given in~\eqref{eq:inflated 2}.
\begin{thm}\label{teor:main 2}
Under Assumptions~{\rm \ref{asu:1}} and~{\rm\ref{asu:2}}, for each $\ep>0$  let $(x_\ep(t),y_\ep(t))$ be the solution of~\eqref{eq:slowfastent}-\eqref{eq:initial values}. Then, given $t_0>0$, for $\ep$ small enough, $(x_\ep(t),y_\ep(t))$ can be extended to $t\in [0,t_0]$. Furthermore:
\begin{itemize}
\item[\rm{(i)}]  Every sequence $\ep_j\to 0$ has a subsequence, say $\ep_k\downarrow 0$, such that $x_{\ep_k}(t)$ converges uniformly for $t\in [0,t_0]$ to a solution of the differential inclusion~\eqref{eq:diff incl}.
Let the limit be $x_0(t)$, for $t\in [0,t_0]$.
\item[\rm{(ii)}] For the fast variables in the fast timescale $y_{\ep_k}(\tau)$, the behaviour is as follows. For a fixed small $r>0$, let $K_0$ be the compact tubular set defined in~\eqref{eq:K0},
let $\A^x\subset \mathcal{H} \times \R^m$  be the global attractor  for the skew-product semiflow~\eqref{eq:sk-pr x} for each fixed $x\in K_0$ and suppose that Assumption~{\rm\ref{assu:cont A^x}} holds. Then, given $\delta >0$, there exist $T=T(\delta,y_0,\{\A^x\}_{x\in K_0})>0$ and $k_0=k_0(\delta,T)>0$ with $2T<t_0/\ep_{k_0}$ such that for every $k\geq k_0$,
\begin{equation*}
 {\rm dist}\big(y_{\ep_k}(\tau),A_{g{\cdot}\tau}^{x_0(\ep_k\tau)}[\delta]\big)\leq  \delta \quad \hbox{for all }\; \tau\in [T,t_0/\ep_k]\,.
\end{equation*}
\end{itemize}
\end{thm}
\begin{proof}
The statement in (i) is exactly the same as in Theorem~\ref{teor:main}. For the proof of (ii) we refer the reader to the proof of  Theorem~\ref{teor:main}~(ii) and we only indicate the necessary changes in the present situation.

Given $0<\delta<1$, and the compact set  $D:=\overline B(0,\rho)\subset \R^m$ built in the cited proof,  it follows from~\eqref{eq:unif atract 2} in Theorem~\ref{teor:param inflated pullb att} that there is a $T=T(D,\delta)$ such that for every  $(h,x)\in\mathcal{H}\times K_0$,
\begin{equation*}
   {\rm dist}\big(y(\tau,h,x,D),A_{h{\cdot}\tau}^x[\delta]\big)<  \frac{\delta}{2}
\end{equation*}
for every $\tau\geq T$. This time we note  that in particular,
\begin{equation*}
   {\rm dist}\big(y(\tau,g{\cdot}\gamma,{x_0(\ep_k s)},D),A_{g{\cdot}(\gamma+\tau)}^{x_0(\ep_k s)}[\delta]\big) < \frac{\delta}{2}
\end{equation*}
 for every $\gamma\in \R$, $k\geq 1$ and every $s,\tau\in [0,t_0/\ep_k]$, provided that $\tau\geq T$.
From here the proof goes smoothly following the same arguments as in the proof mentioned above.  The proof is finished.
\end{proof}

As a direct corollary of this theorem, we have the following result in the line of Theorem~\ref{teor:valor final}.
\begin{cor}\label{coro 1}
In the same situation as in Theorem~{\rm \ref{teor:main 2}}, under Assumptions~{\rm \ref{asu:1}} and {\rm\ref{asu:2}},  assume that $x_{\ep_k}(t)$ converges  to $x_0(t)$ uniformly on $[0,t_0]$, let the compact set $K_0$ and the global attractors $\{\A^x\mid x\in K_0\}$ be the ones defined therein, and suppose that Assumption~{\rm\ref{assu:cont A^x}} holds.  For each integer $k\geq 1$, let
\begin{equation*}
 \delta_k:=\inf\big\{ \delta>0 \mid  {\rm dist}\big(y_{\ep_k}(t_0/\ep_k),A_{g{\cdot}(t_0/\ep_k)}^{x_0(t_0)}[\delta]\big)\leq  \delta  \big\}\,.
\end{equation*}
 Then:
 \begin{itemize}
   \item[\rm{(i)}] $\des\lim_{k\to \infty} \delta_k =0$;
   \item[\rm{(ii)}] $\des\lim_{k\to \infty}  {\rm dist}\big(y_{\ep_k}(t_0/\ep_k),A_{g{\cdot}(t_0/\ep_k)}^{x_0(t_0)}[\delta_k]\big) =0$.
 \end{itemize}
\end{cor}
Before we proceed, let us clarify that in general we cannot disregard the inflated pullback attractors, that is, in general  $[\delta]$ cannot be omitted in the previous statements. A good reference to check all the details of the following example is the paper by Langa et al.~\cite{paper:laos}.
\begin{eje}
Let us consider for each $\ep>0$ a slow-fast system in the fast scale of time $\tau\in [0,t_0/\ep]$,
\begin{equation*}
\left\{\begin{array}{l} x'  = \ep\,f(x,y)\,,
  \\
 y' = a_0(\tau)\,y + k(y)\,,
\end{array}\right.
\end{equation*}
where $f(x,y)$ satisfies the conditions imposed in Section~\ref{secSlowfast} and $g(x,y,\tau)=a_0(\tau)\,y + k(y)$ is a real map with no dependence  on $x$. We assume that $a_0:\R\to \R$ is an almost periodic map with mean value $0$, that is, $0=\lim_{\tau\to \infty}(1/\tau)\int_0^\tau a_0(s)\,ds$, and unbounded integral, $\sup\big\{\big|\int_0^\tau   a_0(s)\,ds\big|\mid  \tau\in\R\big\}=\infty$ (note that this excludes the periodic case). We build the hull of $a_0$, $\mathcal{H}\subset C(\R)$, with the shift flow denoted by $h{\cdot}\tau$ and consider the continuous map $a:\mathcal{H}\to \R$ defined by $h\mapsto h(0)$, in such a way that $a(h{\cdot}\tau)=h(\tau)$ for $\tau\in \R$ and  in particular for $h=a_0$ we recover the initial almost periodic map $a_0$. Then, by Birkhoff's ergodic theorem, we have that $a\in C_0(\mathcal{H})=\left\{r\in C(\mathcal{H})\mid \int_{\mathcal{H}} r\,d\nu=0\right\}$ where $\nu$ is the only invariant measure for the flow in the hull. More precisely, $a\in \mathcal{U}(\mathcal{H})$, which is the subset of $C_0(\mathcal{H})$ of maps with unbounded integral.

We also assume that the scalar family of equations $y' = a(h{\cdot}\tau)\,y + k(y)$, $h\in\mathcal{H}$ is linear-dissipative, that is,  the nonlinear term $k(y)$, which for simplicity is taken to be an odd function, satisfies that $k\in C^1(\R)$, $k(y)=0$ if $|y|\leq r_0$ for a certain $r_0>0$, $y\, k(y)\leq 0$ for all $y\in\R$ and $\lim_{y\to \infty} k(y)/y=-\infty$.

With the former assumptions,  $g$ is an admissible map, locally Lipschitz in $y$ and the skew-product semiflow defined on $\mathcal{H}\times \R$ by the solutions of the family of scalar ODEs $y'= a(h{\cdot}\tau)\,y + k(y)$, $h\in\mathcal{H}$ is globally defined. Besides, it admits a global attractor $\A$ which can be written as $\A = \bigcup_{h\in \mathcal{H}} \{h\}\times A_h \subseteq\bigcup_{h\in \mathcal{H}} \{h\}\times [-b(h),b(h)] $ for $b(h)=\sup A_h$ for each $h\in \mathcal{H}$. This map $b$ plays a major role in what follows. In fact, the hull splits into two nonempty invariant subsets $\mathcal{H}=\mathcal{H}_{\rm s}\cup \mathcal{H}_{\rm f}$ as follows:
\begin{align*}
  \mbox{a residual set } \mathcal{H}_{\rm s} & := \{h\in \mathcal{H}\mid b(h)=0\}=\{h\in \mathcal{H}\mid A_h=\{0\}\} \,,\\
  \mbox{and a dense set of first category } \mathcal{H}_{\rm f} & := \mathcal{H}\setminus \mathcal{H}_{\rm s}=\{h\in \mathcal{H}\mid b(h)>0\}\,.
\end{align*}

These two sets exclusively depend on the linear family $y'= a(h{\cdot}\tau)\,y $, $h\in\mathcal{H}$. Namely, for the associated smooth cocycle $c(\tau,h)=\exp \int_0^\tau a(h{\cdot}s)\,ds$, it holds that $h\in \mathcal{H}_{\rm f}$ if and only if $\sup_{\tau\leq 0}c(\tau,h)<\infty$.
We now introduce a third relevant set. A map $h\in \mathcal{H}$ is said to be (Poincar\'{e}) {\em recurrent\/} at $\infty$ if  there exists a sequence $\{\tau_k\}_k\uparrow \infty$ such that $\lim_{k\to\infty} c(\tau_k,h)=1$. Then, the set
\[
\mathcal{H}_{\rm r}^+ :=\big\{h\in \mathcal{H}\mid h{\cdot}\tau   \;\hbox{is recurrent at}\;\infty\;\hbox{for all}\; \tau\in\R \big\}
\]
is invariant and of full measure.
The thing is that   for all  $h\in \mathcal{H}_{\rm{s}}^*:=\mathcal{H}_{\rm{s}}\cap \mathcal{H}_{\rm{r}}^+$ the corresponding process has no forward attractor (see~\cite[Proposition 4.10]{paper:laos}), that is, for all  $h\in \mathcal{H}_{\rm{s}}^*$ it is not true that
\begin{equation*}
\lim_{\tau\to \infty}{\rm dist}\big(y(\tau,h,D),A_{h{\cdot}\tau}\big)=  0   \quad \hbox{for each bounded set}\; D\subset  \R\,.
\end{equation*}
Note that, as stated before, if $h\in \mathcal{H}_{\rm{s}}^*$,  $A_{h{\cdot}\tau}$ reduces to $\{0\}$ for all $\tau\geq 0$ by the invariance of the set $\mathcal{H}_{\rm{s}}^*$.
In fact, for each $h\in \mathcal{H}_{\rm{s}}^*$ and every $y_0\neq 0$,  the limit $\lim_{\tau\to \infty}|y(\tau,h,y_0)|$ is not null, for if the limit were null, by the linear-dissipative structure of the equation, then $y(\tau,h,y_0)$ would be a solution of the linear equation $y'= a(h{\cdot}\tau)\,y $ from one time on, but this is not possible because $h\in \mathcal{H}_{\rm{r}}^+$.

Then, if $a_0$ itself is in $\mathcal{H}_{\rm{s}}^*$ we just consider the initial slow-fast pair. In general we fix an $h\in \mathcal{H}_{\rm{s}}^*$ and a value $y_0\neq 0$ so that the limit $\lim_{\tau\to \infty}|y(\tau,h,y_0)|$ is not null and we look at the slow-fast equations
\begin{equation}\label{eq:ejemplo}
\left\{\begin{array}{l} x'  = \ep\,f(x,y)\,,
  \\
 y' = h(\tau)\,y + k(y)\,.
\end{array}\right.
\end{equation}
Note that, given initial conditions $(x_0,y_0)$, the solution for each $\ep>0$ is given by $(x_\ep(\tau),y_\ep(\tau))$ where $y_\ep(\tau)$ does not depend on $\ep$, as it is just the solution $y(\tau,h,y_0)$ of the scalar initial value problem
\begin{equation*}
\left\{\begin{array}{l} y' = h(\tau)\,y + k(y)\,,
  \\
 y(0)=y_0\,.
\end{array}\right.
\end{equation*}
We can apply the previous results Theorems~\ref{teor:main} and~\ref{teor:main 2} to the problem in~\eqref{eq:ejemplo} noting that there is no need to consider the compact set $K_0$ because the fast variable does not depend on the parameter $x$. Then, given a sequence $\ep_j\to 0$ there is a subsequence $\ep_k\downarrow 0$ such that,   if the inflated pullback attractor $A_{h{\cdot}\tau}[\delta]$ were dispensable in~\eqref{eq:teorema}, we would have that, for each $\delta>0$ there exists a $T>0$ and an integer $k_0$ such that for all $k\geq k_0$,  $|y_{\ep_k}(\tau)|=|y(\tau,h,y_0)|< \delta$ for all $\tau\in [T,t_0/\ep_k]$, but this cannot happen if $\lim_{\tau\to \infty}|y(\tau,h,y_0)|$ is not null. Therefore, one cannot do without $A_{h{\cdot}\tau}[\delta]$ in this example.

Let us mention that specific examples of almost periodic maps $a_0$ satisfying $a_0\in \mathcal{H}_{\rm{s}}$ or $a_0\in \mathcal{H}_{\rm{f}}$, among other conditions,  have been reviewed  in~\cite[Example~4.15]{paper:laos}. To finish, it is worth noting that for a minimal uniquely ergodic and aperiodic flow $(\mathcal{H},{\cdot})$,   the fact that
$\nu(\mathcal{H}_{\rm{s}})=1$ is generic in $\mathcal{U}(\mathcal{H})$. Whenever $a\in \mathcal{U}(\mathcal{H})$ and  $\nu(\mathcal{H}_{\rm{s}})=1$, then the set of maps $\mathcal{H}_{\rm{s}}^*=\mathcal{H}_{\rm{s}}\cap \mathcal{H}_{\rm{r}}^+$ providing examples of scalar equations which do not have forward attraction is a residual invariant set of full measure. On the other hand, if for $a\in\mathcal{U}(\mathcal{H})$ it is $\nu(\mathcal{H}_{\rm{f}})=1$, then the global attractor is fiber-chaotic in measure in the sense of Li-Yorke. More examples of nonautonomous linear-dissipative or purely dissipative scalar ODEs exhibiting characteristics of chaotic dynamics in $\mathcal{H}\times \R$ can be found in the works~\cite{ObSa,CaNuOb}.
\end{eje}
Let us now take a final step  in  strengthening the required conditions in order to improve the delimitation of the approximation zone for the fast motion in $\R^m$ as $\ep\to 0$, making it easier to understand in applications to real problems. More precisely, in the next result we will add Assumption~\ref{assu:cont A_(h,x)} on the continuity of the section map $(h,x)\in \mathcal{H}\times K_0\mapsto A_{(h,x)}\subset \R^m$. Note that this map is well-known to be upper semi-continuous  (once more, see, e.g., Theorem~3.34 in~\cite{klra}), so that the real assumption is on the lower semi-continuity.
In this situation there is no need to consider the parametrically inflated pullback attractors, thanks to~\eqref{eq:unif forw} in Theorem~\ref{teor:param inflated pullb att}, and the result reads as follows. Once more, we include the statement in (i) for the sake of completeness.
\begin{thm}\label{teor:main 3}
Under Assumptions~{\rm \ref{asu:1}} and {\rm\ref{asu:2}}, for each $\ep>0$  let $(x_\ep(t),y_\ep(t))$ be the solution of~\eqref{eq:slowfastent}-\eqref{eq:initial values}. Then, given $t_0>0$, for $\ep$ small enough, $(x_\ep(t),y_\ep(t))$ can be extended to $t\in [0,t_0]$. Furthermore:
\begin{itemize}
\item[\rm{(i)}]  Every sequence $\ep_j\to 0$ has a subsequence, say $\ep_k\downarrow 0$, such that $x_{\ep_k}(t)$ converges uniformly for $t\in [0,t_0]$ to a solution of the differential inclusion~\eqref{eq:diff incl}.
Let the limit be $x_0(t)$, for $t\in [0,t_0]$.
\item[\rm{(ii)}] For the fast variables in the fast timescale $y_{\ep_k}(\tau)$, the behaviour is as follows. For a fixed small $r>0$, let $K_0$ be the compact tubular set in~\eqref{eq:K0}, let $\A\subset \mathcal{H} \times K_0 \times \R^m$ be the global attractor of the  skew-product semiflow~\eqref{eq:sk-pr} and suppose that Assumption~{\rm\ref{assu:cont A_(h,x)}} holds.  
Then, given $\delta >0$, there exist $T=T(\delta,y_0,\A)>0$ and $k_0=k_0(\delta,T)>0$ with $2T<t_0/\ep_{k_0}$ such that for every $k\geq k_0$,
\begin{equation*}
 {\rm dist}\big(y_{\ep_k}(\tau),A_{(g{\cdot}\tau,x_0(\ep_k\tau)}\big)\leq  \delta \quad \hbox{for all }\; \tau\in [T,t_0/\ep_k]\,.
\end{equation*}
\item[\rm{(iii)}] $ \des\lim_{k\to\infty} {\rm dist}\big(y_{\ep_k}(t_0/\ep_k),A_{(g{\cdot}(t_0/\ep_k),x_0(t_0))}\big)=0\,.$

\item[\rm{(iv)}] In the particular case when the global attractor $\A$ is a copy of the base, that is, when there is a continuous map $\eta(h,x)$ such that $A_{(h,x)}=\{\eta(h,x)\}$ for each $(h,x)\in\mathcal{H}\times K_0$,  given $\delta >0$, there exist $T=T(\delta,y_0,\A)>0$ and $k_0=k_0(\delta,T)>0$ with $2T<t_0/\ep_{k_0}$ such that for every $k\geq k_0$,
\begin{equation}\label{eq:teorema copy}
  \big|y_{\ep_k}(\tau)- \eta\big(g{\cdot}\tau,x_0(\ep_k\tau)\big)\big|\leq  \delta \quad \hbox{for all }\; \tau\in [T,t_0/\ep_k]\,.
\end{equation}
Besides, $\des\lim_{k\to\infty} \big|y_{\ep_k}(t_0/\ep_k) - \eta\big(g{\cdot}(t_0/\ep_k),x_0(t_0)\big)\big| =0$ and, if $g{\cdot}(t_0/\ep_k)\to h_0$ in $(\mathcal{H},d)$ as $k\to\infty$, then $h_0\in \Om(g)$ and $\des\lim_{k\to\infty} y_{\ep_k}(t_0/\ep_k) = \eta\big(h_0,x_0(t_0)\big)$.
\end{itemize}
\end{thm}
\begin{proof}
We omit the proof of (ii). Let us just mention that~\eqref{eq:unif forw} in Theorem~\ref{teor:param inflated pullb att} is used. The fact that the limit in (iii) is null  follows from (ii) by evaluating at the final time  $t_0/\ep_k$.  As for (iv), note that the continuity condition $(h,x)\in \mathcal{H}\times K_0\mapsto A_{(h,x)}=\{\eta(h,x)\}\subset\R^m$ in Assumption~{\rm\ref{assu:cont A_(h,x)}} is satisfied, so that we have  rewritten the former conclusions in (ii)-(iii) in this particular situation. Note that, if $g{\cdot}(t_0/\ep_k)\to h_0$ (this happens at least for a subsequence), then there exists the limit $\lim_{k\to\infty} y_{\ep_k}(t_0/\ep_k) = \eta\big(h_0,x_0(t_0)\big)$. The proof is finished.
\end{proof}
\begin{rmk}
When Theorem~\ref{teor:main 3} applies, it follows that for $\ep>0$ small enough the fast motion $y_\ep(\tau)$ varies in the fast timescale in a way that it remains close to the fibers of the attractors $\A^{x_0(\ep\tau)}$ along the limit of the slow motion, along the orbit of $g$, that is, $A^{x_0(\ep\tau)}_{g{\cdot}\tau}$, for long intervals of time in $(0,t_0/\ep]$. This fact is implicit in classical references that deal with the particular case in which the layer equations of an autonomous slow-fast pair admit an $x$-parametric family of periodic solutions with asymptotical orbital stability. For instance, see Pontryagin and Rodygin~\cite{paper:PR}.
\end{rmk}
\subsection{Numerical simulations in a simple case}\label{sec:subsec numeric} In this subsection, we present some numerical examples to showcase the impact of the previous results in the understanding of the dynamics  of singularly perturbed nonautonomous differential equations.

Let us consider a slow-fast pair of ODEs in the fast scale of time, as in~\eqref{eq:slowfastentau}, with a nonautonomous variation in the fast motion, which initially does not depend on the slow motion $x$,
\begin{equation*}
\begin{cases}
x'=\ep\,f(x,y)\,,\\
y'=g(y,\tau)\,,\\
\end{cases}
\end{equation*}
where $\ep>0$ is a small parameter. It is assumed that $f:\R^n\times\R^m\to \R^n$ and $g:\R^m\times\R\to\R^m$ satisfy the conditions in Assumptions~\ref{asu:1} and~\ref{asu:2} (i) and (ii), meaning in this situation that the solutions of $y'=g(y,\tau)$ are uniformly ultimately bounded. Let $\mathcal{H}$ be the hull of $g$. Then, it follows from Proposition~\ref{prop:existe atractor} that there is a global attractor $\A_0=\bigcup_{h\in\mathcal{H}}\{h\}\times A_h\subset \mathcal{H}\times\R^m$ for the skew-product semiflow given by the solutions of the family of equations $y'=h(y,\tau)$ for  $h\in\mathcal{H}$. If $\Gamma:\R^n\to \R^m$ is a continuous map, let us introduce in the fast motion a dependence on the slow motion through the map $\Gamma$, that is, let us consider the slow-fast pair
\begin{equation*}
\begin{cases}
x'=\ep\,f(x,y)\,,\\
y'=g\big(y-\Gamma(x),\tau\big)\,.\\
\end{cases}
\end{equation*}
Due to the continuity of $\Gamma$, it is immediate to check that $\tilde g(x,y,\tau):=g\big(y-\Gamma(x),\tau\big)$ satisfies Assumption~\ref{asu:1} (ii)-(iii). Besides, it is also easy to check that for each fixed $x\in\R^n$, the translated set $\A^x=\bigcup_{h\in\mathcal{H}}\{h\}\times \big(A_h+\Gamma(x)\big)\subset \mathcal{H}\times\R^m$ is the global attractor for the skew-product semiflow $\pi^x$. This implies in particular that the solutions of $y'=g\big(y-\Gamma(x),\tau\big)$ are uniformly ultimately bounded. By the way the fibers $A^x_h$ are translations of the fibers $A_h$ by the vector $\Gamma(x)$, given a compact set $K_0\subset \R^n$, by  Proposition~\ref{prop:unifom ulti bound on K} we can conclude that Assumption~\ref{asu:2} holds. Since Assumption~\ref{assu:cont A^x} also holds, Theorem~\ref{teor:main 2} on the tracking of the fibers of the inflated pullback attractors $A^{x_0(\ep\tau)}_{g{\cdot}\tau}[\delta]$ and its Corollary~\ref{coro 1} apply in this situation. Whenever Assumption~\ref{assu:cont A_(h,x)} holds too, then Theorem~\ref{teor:main 3} applies, with no need to consider inflated pullback attractors, but just pullback attractors for the tracking. Note that no smooth conditions are required on $f$, $g$ or $\Gamma$.

For the numerical simulations, we choose a simple quadratic equation within the previous setting, but with the particularity that the required  conditions are satisfied only locally. Let us explain in detail the example. We consider a slow-fast pair of scalar ODEs by taking in the above model   $f(x,y)\equiv 1$ and $g(y,\tau)=-y^2+p(\tau)$ for a bounded and uniformly continuous map $p:\R\to\R$,
\begin{equation}\label{eq:RiccatiSF}
\begin{cases}
x'=\ep\,,\\
y'=-\big(y-\Gamma(x)\big)^2+p(\tau)\,,\\
\end{cases}
\end{equation}
where $\Gamma:\R\to\R$ is a continuous map and $\ep>0$ is a small parameter.

Note that, for each fixed $x\in \R$, we have a scalar nonautonomous Riccati equation for the fast variable:
\begin{equation}\label{eq:riccati}
y'=-\big(y-\Gamma(x)\big)^2+p(\tau)\,.
\end{equation}
The dynamical scenarios of~\eqref{eq:riccati} have been investigated in detail in Longo et al.~\cite{lnor} and Longo et al.~\cite{lno}. Only one of the following three cases is possible:
\begin{itemize}
  \item[a)] \eqref{eq:riccati} admits an attractor-repeller pair of uniformly separated hyperbolic solutions;
  \item[b)] \eqref{eq:riccati} has bounded solutions but no hyperbolic ones;
  \item[c)] \eqref{eq:riccati} has no bounded solutions.
\end{itemize}
Recall that a bounded solution $  b\colon\R\to\R$ of~\eqref{eq:riccati} is said to be
{\em hyperbolic\/} if the corresponding variational equation $z'=-2\big(b(\tau)-\Gamma(x)\big)\,z$
has an exponential dichotomy on $\R$. That is,  if
there exist $k_b\ge 1$ and $\beta_b>0$ such that either
\begin{equation}\label{eq:masi}
 \exp\int_s^\tau-2\big(b(u)-\Gamma(x)\big)\,du\le k_b\,e^{-\beta_b(\tau-s)} \quad
 \text{whenever $\tau\ge s$}\,
\end{equation}
or
\begin{equation}\label{eq:menosi}
 \exp\int_s^\tau -2\big(b(u)-\Gamma(x)\big)\,du\le k_b\,e^{\beta_b(\tau-s)} \quad
 \text{whenever $\tau\le s$}
\end{equation}
holds. In the case~\eqref{eq:masi}, the hyperbolic solution $  b$ is said to be
{\em (locally) attractive}, and in the case~\eqref{eq:menosi}, $  b$ is {\em (locally) repulsive}. An hyperbolic attractive solution is both locally pullback attractive and locally forward attractive. If an attractor-repeller pair exists for~\eqref{eq:riccati}, then the attractor is strictly above the repeller and they completely describe the asymptotic behaviour of any other solution, in that solutions above the attractor will converge to it in forward time, solutions below the repeller will converge to it in backward time and solutions in between the attractor and the repeller will converge to the former in forward time and to the latter in backward time.

For the computations, we shall  consider the quasi-periodic map
\begin{equation}\label{eq:p}
p(\tau)=-\sin(\tau/2)-\sin(\sqrt{5}\,\tau)+0.962\,,
\end{equation}
and work under the assumption that case a) holds true for $y'=-y^2+p(\tau)$: this is supported by numerical evidence as shown in Figure~\ref{fig:assumptions}.
\begin{figure}
    \centering
\begin{overpic}[abs,unit=0.5mm,scale=.25, width=0.95\textwidth]{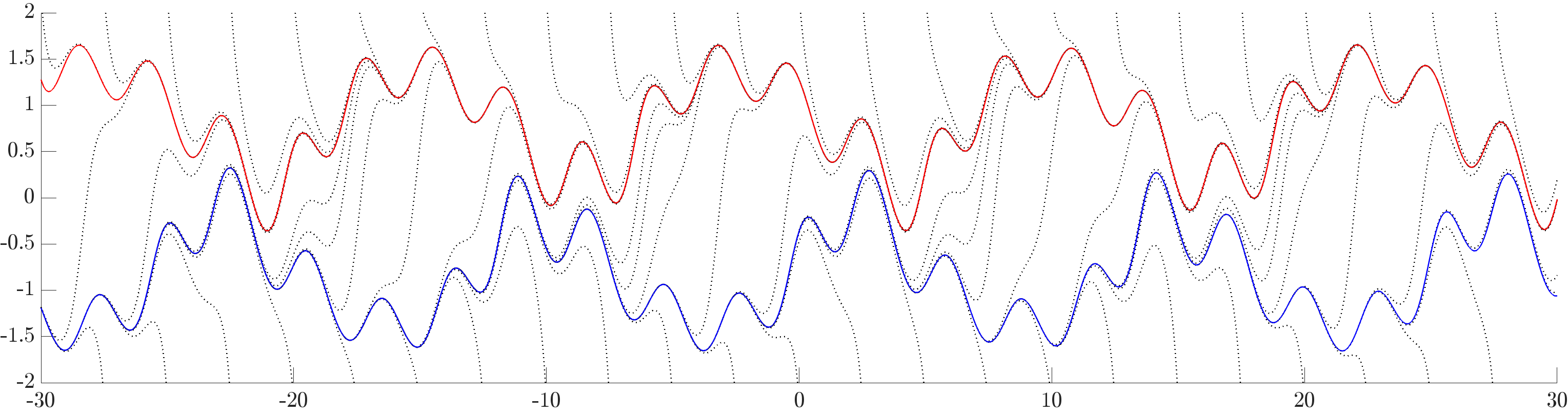}
\put(-5,33){$y$}
\put(121,-5){$\tau$}
\end{overpic}
\caption{The attractor (in red) and the repeller (in blue) for $y'=-y^2+p(\tau)$ with $p$ as in~\eqref{eq:p} together with sample solutions (dotted lines) to highlight the qualitative behaviour.}
\label{fig:assumptions}
\end{figure}
In such a case every equation of the type $y'=-(y-\gamma)^2+p(\tau)$, with $\gamma\in\R$, also has an attractor-repeller pair, which is in fact a vertical translation by $\gamma$ of the attractor-repeller pair obtained for $\gamma=0$. Additionally, when the hull $\mathcal{H}$ of the map $p(\tau)$ is built, the well-known  propagation of the exponential dichotomy through the shift flow on the hull
guarantees that for all $h\in\mathcal{H}$, $y'=-y^2+h(\tau)$ has an attractor-repeller pair (see~\cite{lnor}). In fact, there are two hyperbolic copies of the hull $\mathcal{H}$, $\{(h,r(h))\mid h\in \mathcal{H}\}\subset \mathcal{H}\times \R$ which is a (local) repeller, and $\{(h,a(h))\mid h\in \mathcal{H}\}\subset \mathcal{H}\times \R$ which is a (local) attractor, such that for each $h\in \mathcal{H}$, $\tau\mapsto a(h{\cdot}\tau)$ and $\tau\mapsto r(h{\cdot}\tau)$ is the attractor-repeller pair for the equation $y'=-y^2+h(\tau)$.

Let us explain what this implies for the skew-product flows $\pi^x$ induced by the solutions of the family  $y'=-(y-\Gamma(x))^2+h(\tau)$ for $h\in \mathcal{H}$, for each fixed $x\in \R$:

1. There are two hyperbolic copies of the hull $\mathcal{H}$, $\{(h,r(h)+\Gamma(x))\mid h\in \mathcal{H}\}\subset \mathcal{H}\times \R$ which is a (local) repeller, and $\{(h,a(h)+\Gamma(x))\mid h\in \mathcal{H}\}\subset \mathcal{H}\times \R$ which is a (local) attractor,     given by $\pi^x$-invariant continuous maps.

2. The restriction of the skew-product semiflow $\pi^x$ to the invariant open zone $G(x):=\{(h,y_0)\in \mathcal{H}\times \R \mid y_0> r(h)+\Gamma(x)\}$ has a global attractor, which is precisely $\A^x=\{(h, a(h)+\Gamma(x))\mid h\in \mathcal{H}\}$. On this occasion, the global attractor attracts sets $\mathcal{H}\times X_1\subset G(x)$ for compact sets $X_1$ of $\R$.

3. The solutions of $y'=-(y-\Gamma(x))^2+h(\tau)$, $y(0)=y_0$ for $(h,y_0)\in G(x)$ are  uniformly ultimately bounded.

Although this might require a note explaining some technicalities, we can affirm that, fixed a compact set $K_0\subset\R$,  on the restricted area  \[
\{(h,x,y_0)\in \mathcal{H}\times K_0\times \R \mid y_0>r(h)+\Gamma(x)\}\,,
\]
$\A=\{(h,x,a(h)+\Gamma(x))\mid (h,x)\in \mathcal{H}\times K_0\}\subset \mathcal{H}\times K_0\times\R$ is the global attractor for the skew-product semiflow $\pi$, Assumption~\ref{assu:cont A_(h,x)} holds and Theorem~\ref{teor:main 3}~(iv) applies.

Fixed an interval $[0,t_0]$ in the slow time scale and the initial condition $x(0)=0$, since the map $f$ is independent of $y$, for all $\ep>0$, $x_\ep(t)=t$ assumes values in $K_0=[0,t_0]\subset\R$ and there exists the limit as $\ep\to 0$, $x_\ep(t)\to x_0(t)=t$, for $t\in [0,t_0]$.
Therefore, in accordance with the conclusions of Theorem~\ref{teor:main 3}, we expect that for $\ep>0$ small enough, every solution $(x_\ep(t),y_\ep(t))$ of~\eqref{eq:RiccatiSF} can be extended to $t\in [0,t_0]$ and the tracking of $\eta(p{\cdot}\tau,x_0(\ep\tau)):=a(p{\cdot}\tau)+\Gamma(\ep\tau)$ stated in~\eqref{eq:teorema copy} happens. To show this numerically, we shall firstly fix
\[
\Gamma(x)=0.2\,\big(\sin(\sqrt{2}\,x)+\cos(x/5)\big)\,,\quad x\in\R\,.
\]
Note that for each $\ep>0$, $y_\ep(\tau)$ is the solution of the equation
\begin{equation}\label{eq:riccati ep}
y'=-\big(y-\Gamma(\ep\tau)\big)^2+p(\tau)\,
\end{equation}
with a fixed initial condition $y(0)=y_0$. The dynamical scenarios of this equation are the same as those of~\eqref{eq:riccati}  and it  has an attractor-repeller pair, provided that $\ep$ is small enough (see~\cite{lnor}). We denote by $a_\ep(\tau)$, $\tau\in\R$ the hyperbolic (local) attractor  towards which all the solutions starting above the repeller converge in forward time. So, the tracking will be evident if, for $\ep$ small enough, this curve remains sufficiently close to $\eta(p{\cdot}\tau,x_0(\ep\tau))=a(p{\cdot}\tau)+\Gamma(\ep\tau)$. This is what we show in Figure~\ref{fig:qp-gamma} for some values of $\ep$.

\begin{figure}
    \centering
\begin{overpic}[abs,unit=0.5mm,scale=.25, trim={1.5cm 1.9cm 2cm 1.1cm},clip,width=0.45\textwidth]{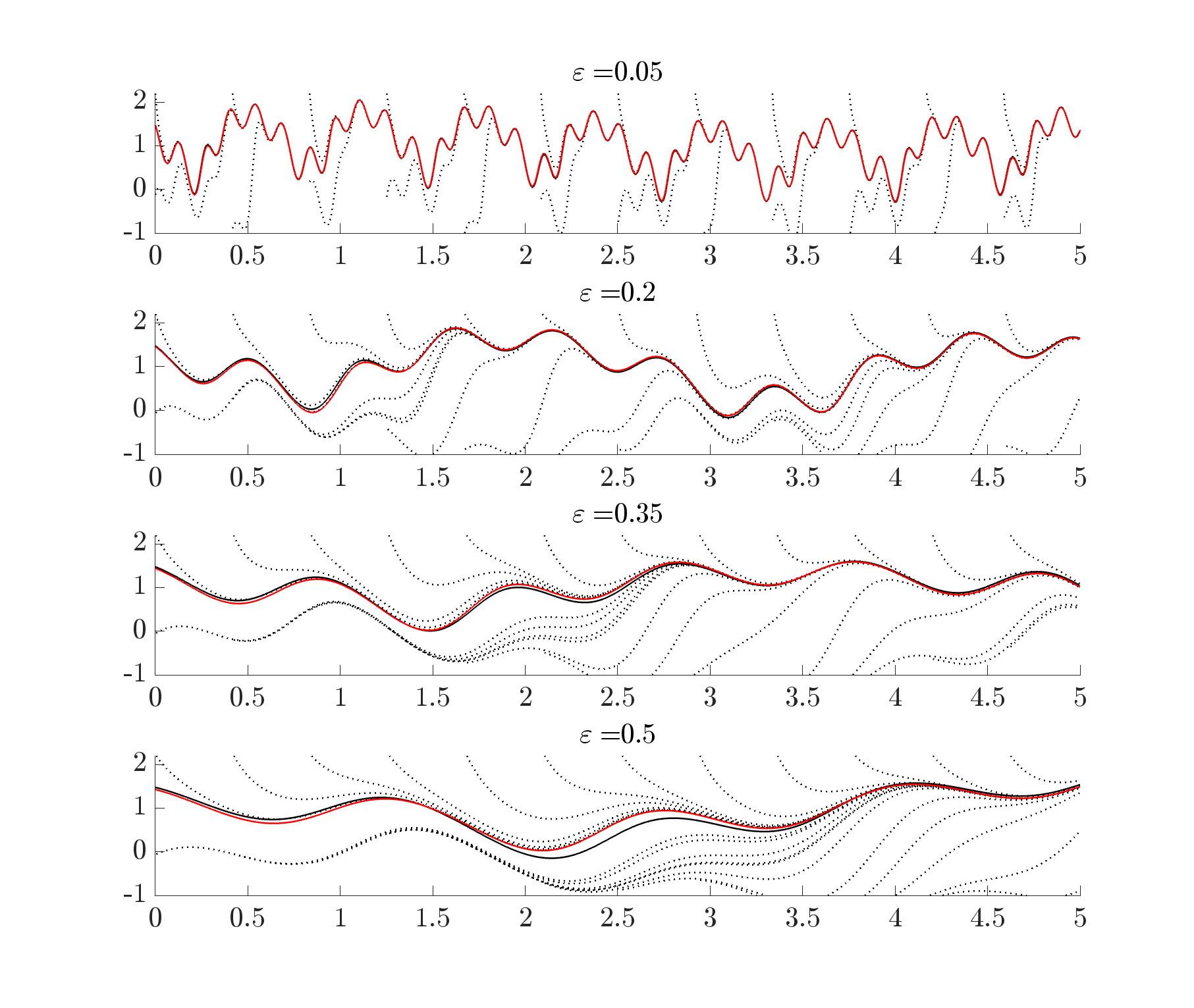}
\put(-5,50){$y$}
\put(60,-5){$t$}
\end{overpic}
\hspace{4mm}
\begin{overpic}[abs,unit=0.5mm,scale=.25, width=0.45\textwidth]{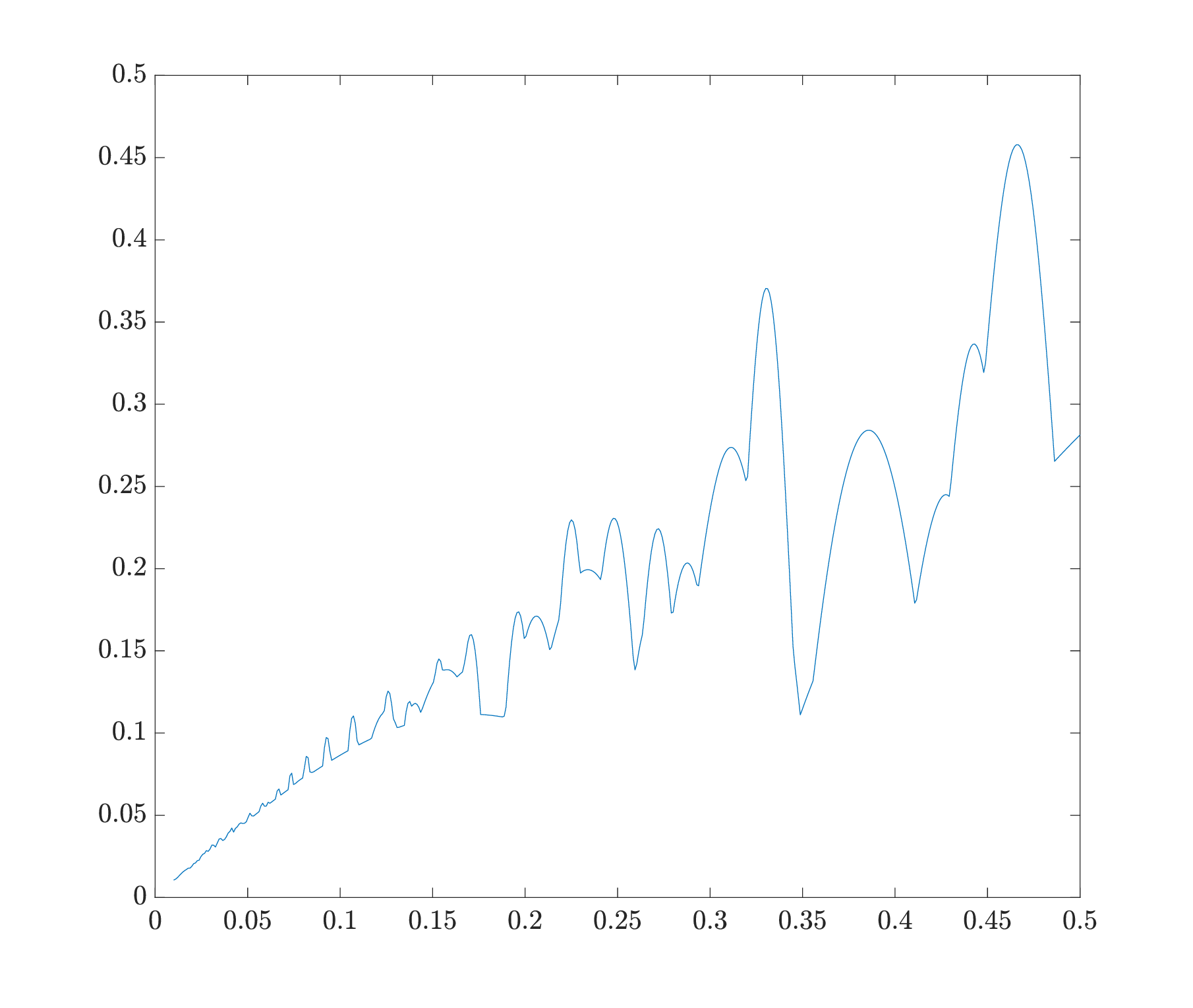}
\put(-8,7){\rotatebox{90}{$\max_\tau|a_\ep(\tau)-\eta(p{\cdot}\tau,\ep\tau)|$}}
\put(60,-5){$\ep$}
\end{overpic}
\caption{In the left panel the graph of $\eta\big(p{\cdot}\tau,\ep\tau\big)$ is depicted in a black solid line  and the hyperbolic (local) attractor $a_\ep(\tau)$ for~\eqref{eq:riccati ep} is depicted in red for $\ep\in\{0.05,0.2,0.35,0.5\}$. The black dotted lines represent sample solutions $y_\ep(\tau)$ starting at different initial conditions showcasing the sought for tracking phenomenon after a certain amount of time $T=T(\delta,y_0,\A)>0$ has passed.
The right panel shows the behaviour of the approximation as $\ep\to0$.
The attractor $a_\ep(\tau)$ is used as a proxy of the approximation provided by $\eta(p{\cdot}\tau,\ep\tau)$ for $\tau\in[0,20/\ep]$ for each $\ep\in (0,0.5]$.
}
\label{fig:qp-gamma}
\end{figure}

\begin{figure}
    \centering
\begin{overpic}[abs,unit=0.5mm,scale=.25, width=0.42\textwidth]{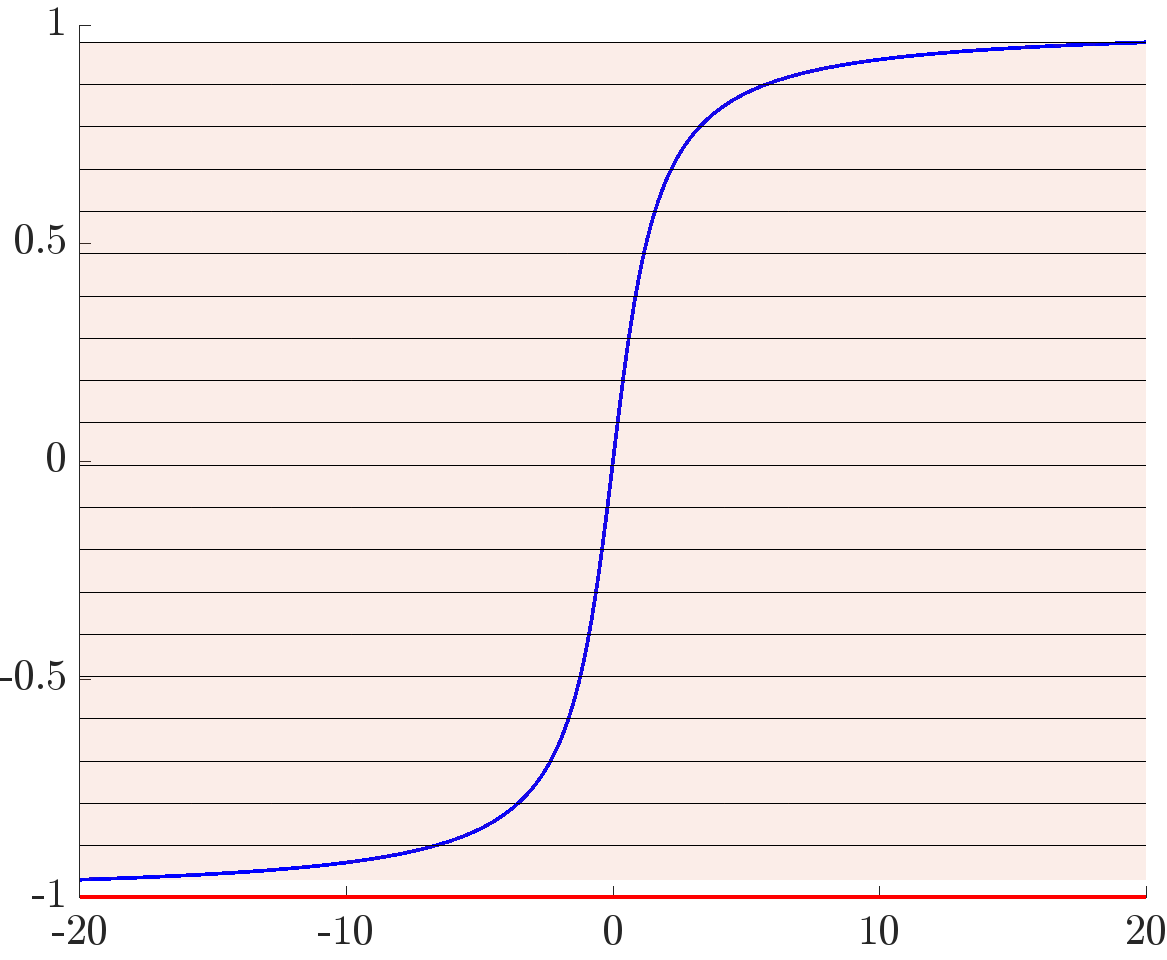}
\put(-5,45){$\gamma$}
\put(53,-5){$\tau$}
\end{overpic}
\hspace{2mm}
\begin{overpic}[abs,unit=0.5mm,scale=.25,  width=0.42\textwidth]{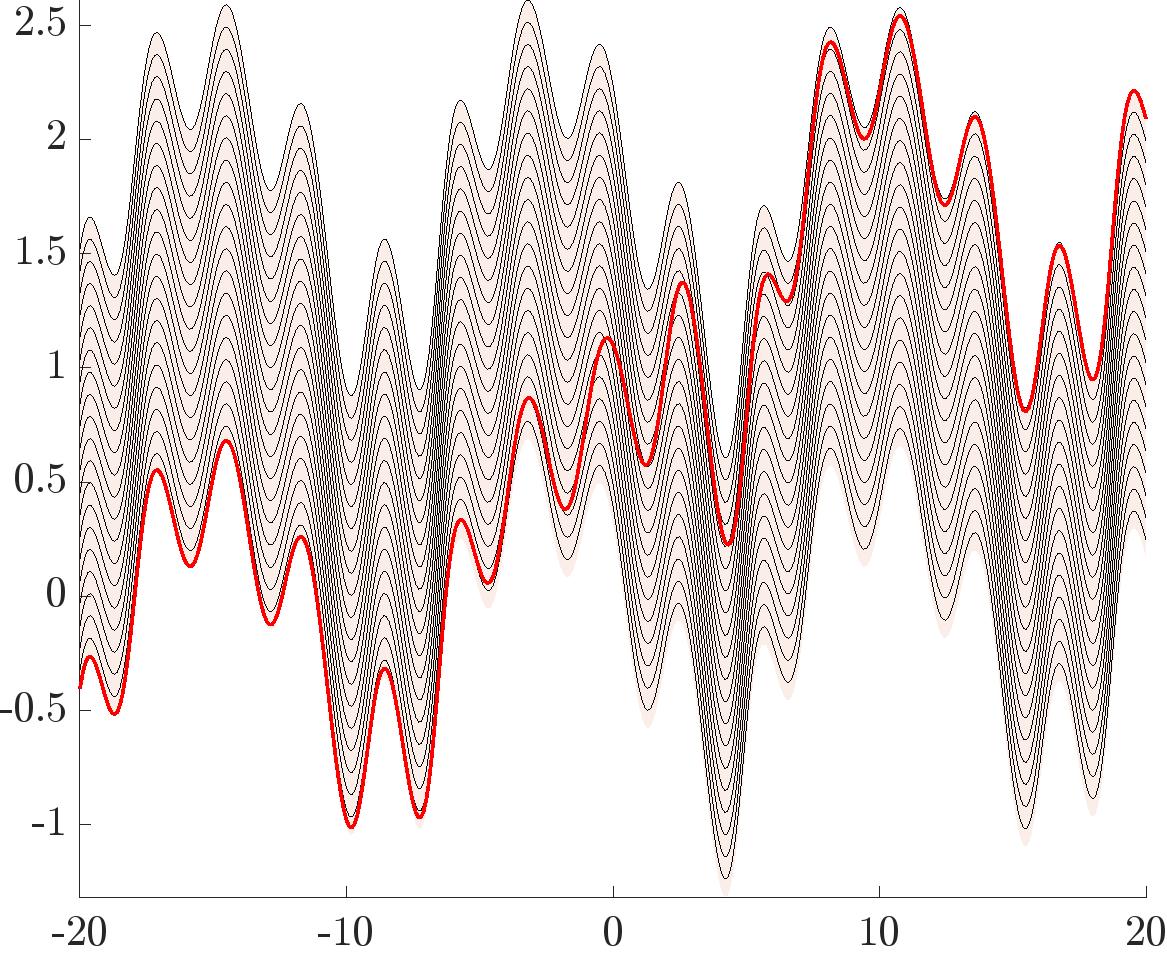}
\put(-5,45){$y$}
\put(53,-5){$\tau$}
\end{overpic}\\
\vspace{5mm}
\begin{overpic}[abs,unit=0.5mm,scale=.25,  width=0.42\textwidth]{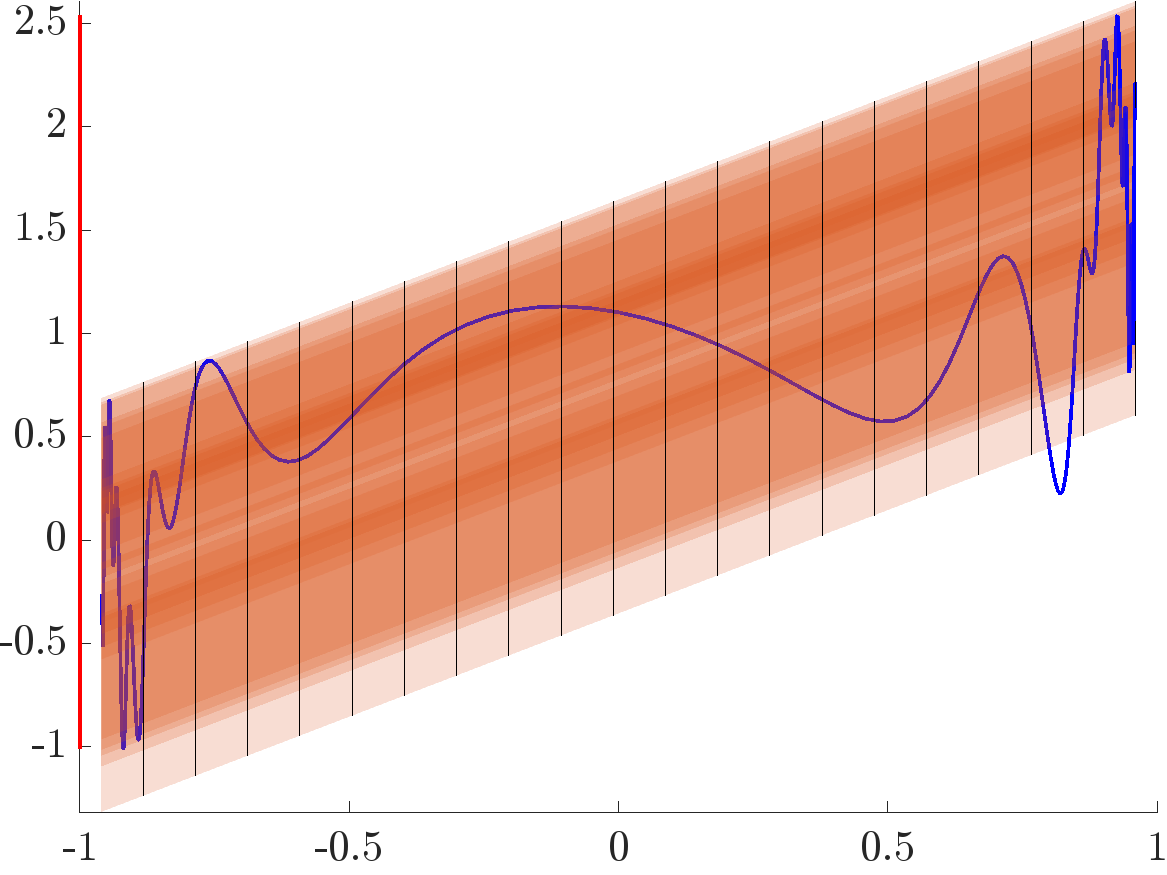}
\put(-5,45){$y$}
\put(53,-5){$\gamma$}
\end{overpic}
\hspace{2mm}
\begin{overpic}[abs,unit=0.5mm,scale=.25, width=0.42\textwidth]{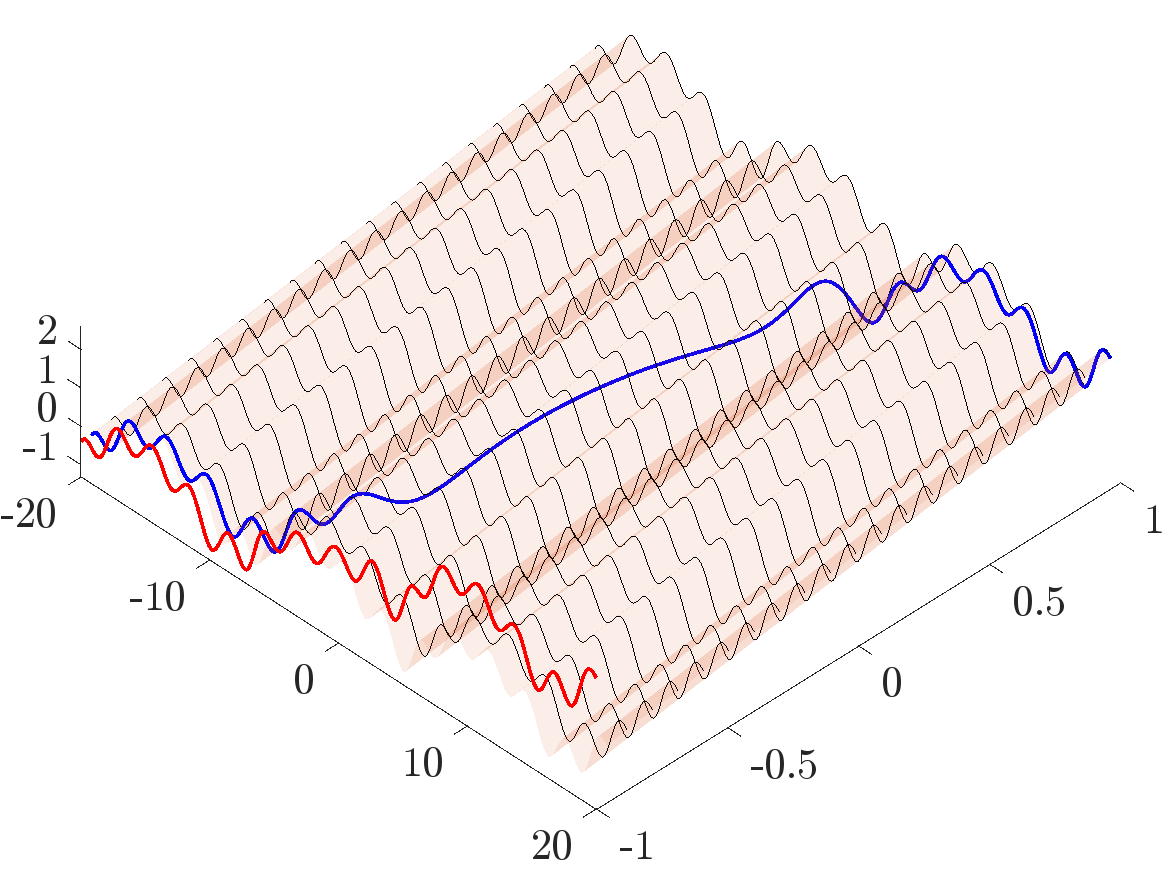}
\put(-5,48){$y$}
\put(20,15){$\tau$}
\put(82,15){$\gamma$}
\end{overpic}
\caption{Tracking in a rate-dependent problem. The attractor (in red) of~\eqref{eq:riccati ep} with $\Gamma(x)=2/\pi\, \arctan (x)$ and $\ep=0.8$ tracks the parametric family of local attractors for $y'=-(y-\gamma)^2+p(\tau)$ and $\gamma\in (-1,1)$ (orange surface layered by black curves for a sample of values in $(-1,1)$). The blue curve is the projection of the red attractor onto the orange surface. Four different point of views are shown. }
\label{fig:rate-tracking}
\end{figure}

As a second example, we consider the map $\Gamma(x)=2/\pi\, \arctan (x)$, $x\in\R$  which is asymptotically constant, i.e., there exist the limits $\lim_{x\to\pm\infty} \Gamma(x)=\pm 1$. In this case, for each $\ep>0$, $y_\ep(\tau)$ is a solution of the differential equation~\eqref{eq:riccati ep} which models a transition from the past equation
\[
y'=- (y+1)^2 +p(\tau)
\]
to the future equation
\[
y'=- (y-1)^2 +p(\tau)\,,
\]
and the speed of the transition is given by the rate $\ep>0$. This set-up is typically used to investigate a phenomenon of critical transition called rate-induced tipping.  In rate-induced tipping, two asymptotic problems differ only in the value of a parameter but share the same dynamical structure with the presence of a local attractor. A time-dependent variation of the parameter $\Gamma(\ep \tau)$ connects the two problems, respectively as $\tau\to-\infty$ and as $\tau\to\infty$. Typically, an attractor of the transition equation permits to connect the attractors in the past and in the future whenever $\ep$, the rate of transition, is sufficiently small. This phenomenon is also called tracking. In some cases, however, this connection is broken when $\ep$ increases, this is the so-called tipping. For a detailed presentation, we refer the interested reader to the initial work by Ashwin et al.~\cite{paper:APW} with autonomous systems in the past and in the future,  and the subsequent properly nonautonomous works~\cite{lnor,lno,lno2,dlo}, where the quadratic and more general concave problems were thoroughly analysed. Incidentally, these works also clarify that a rate-induced tipping (as well as any other possible critical transition) for this type of problems always corresponds to a nonautonomous saddle-node bifurcation.

It is interesting to note that whenever $p(\tau)$ is constant, the autonomous slow-fast theory on non-compact manifolds by Eldering~\cite{book:Eld} allows to explain the tracking in terms of the persistence of the critical manifold of hyperbolic stable equilibria of the layered problem in the form of a so-called {\it slow manifold}. In the nonautonomous version of these results, however, things are different. On the one hand, the so-called {\it integral manifold} of each layer problem, even when it is a copy of the base (in general a global (pullback) attractor can be much more complicated),  has some  continuity properties but is not necessarily smooth. On the other hand,  the term {\it slow manifold} no longer goes well with the dynamics. The counterpart of the critical manifold is given by the union of the fibers of the pullback attractor and the tracking is observed in the fast scale of time. Now, however, Theorem~\ref{teor:main 3} allows us to explain the tracking phenomenon in a fashion which is similar at least in spirit. The attractor of the transition equation does in fact track the curve that at each instant of time $\tau$ attains the value $\eta(p{\cdot}\tau,\ep\tau)=a(p{\cdot}\tau)+\Gamma(\ep\tau)$ of the pullback and forward (local) attractor of $y'=-(y-\gamma)^2+p(\tau)$ for the value of the parameter $\gamma= \Gamma(\ep\tau)$. Figure~\ref{fig:rate-tracking} aims at showcasing this phenomenon by representing the attractor of the transition equation depending on time (in red) and also projected (in blue) onto the family of attractors of the layer problems depending on the value of the parameter $\gamma$ (orange surface and black curves).
It is important to note that Theorem~\ref{teor:main 3} applies also to higher dimensional problems and attractors which are not a singleton in the fiber.
\par
All the numerical experiments have been carried out with Matlab R2023b and the produced code is available at~\cite{code:LOS}.
\section{A more general version of the results under relaxed assumptions}\label{secOmegalim}
We keep the notation introduced in Sections~\ref{secSlowfast} and~\ref{secTracking} and consider the slow-fast pair~\eqref{eq:slowfastent} with initial conditions~\eqref{eq:initial values}. In this section we maintain Assumption~\ref{asu:1} on the continuity  and limited growth of $f$, and the admissible and $y$-Lipschitz character of $g$, but  we  relax the requirements in  Assumption~\ref{asu:2} into Assumption~\ref{asu:3} in a way  that the base flow $\mathcal{H}$ built associated to the  fast variables is now restricted to $\Om(g)$, the omega-limit set of $g$, which is a compact invariant subset of the hull $\mathcal{H}$. Relaxing Assumption~\ref{asu:2} makes the range of applications clearly wider.

Similar to Definition~\ref{def:ultimately}, we now include the definition of uniform ultimate boundedness of the solutions on an interval $[r_0,\infty)$.
\begin{defn}
Let $K_0\subset \R^n$ and assume that for each $x\in K_0$, $y_0\in \R^m$ and $\tau_0\geq r_0$ the solution  $y(\tau,g,x,\tau_0,y_0)$ is defined on $[\tau_0,\infty)$.  The solutions of the parametric family $y'=g(x,y,\tau)$ for $x\in K_0$ are said to be
{\em uniformly ultimately bounded on $[r_0,\infty)$\/} if there is a constant $c=c(r_0)>0$ such that for every $d>0$ there is a time $T=T(r_0,d)$ such that
\begin{equation*}
  |S_g^x(\tau+\tau_0,\tau_0)\,y_0|=|y(\tau+\tau_0,g,x,\tau_0,y_0)|\leq c
\end{equation*}
for every $x\in K_0$, whenever  $\tau_0\geq r_0$, $\tau\geq T$ and $|y_0|\leq d$.
\end{defn}

\begin{assu}\label{asu:3}
\begin{itemize}
  \item[]
  \item[(i)] For each $h\in \{g{\cdot}\tau_0\mid \tau_0\geq 0\}\cup \Om(g)$, $x\in \R^n$ and $y_0\in \R^m$, the solution  $y(\tau,h,x,y_0)$ is defined for all $\tau\geq 0$.
  \item[(ii)] For each compact set $K_0$ of $\R^n$, the solutions of the parametric family $y'=g(x,y,\tau)$ for $x\in K_0$ are uniformly ultimately bounded on the interval  $[0,\infty)$.
\end{itemize}
\end{assu}
We omit the proof of the following result, since it is part of Theorem~5.10 in~\cite{paper:LNO2}.
\begin{prop}
Under  Assumptions~{\rm\ref{asu:1}~(ii)-(iii)}  and~{\rm\ref{asu:3}},  given a compact set $K_0\subset \R^n$ and a map $h\in \Om(g)$,
 the solutions of the family of equations $y'=h(x,y,\tau)$, $x\in K_0$ are uniformly ultimately bounded.
\end{prop}
As a consequence, under Assumptions~\ref{asu:1} and~\ref{asu:3}, whenever $g\in \Om(g)$, then $\Om(g)=\mathcal{H}$,  Assumption~\ref{asu:2} holds and  the results in the previous section directly apply. This happens, for instance, whenever the flow in the hull is minimal, as it is the case when $g$ is periodic or almost periodic, uniformly on compact sets.  This makes the results  in this section specially relevant when $g\notin\Om(g)$. Note that this entails a much weaker recurrence in time than, e.g., almost periodicity.

Under  Assumptions~{\rm\ref{asu:1}~(ii)-(iii)}  and~{\rm\ref{asu:3}}, fixed a compact set $K_0\subset \R^n$, let us consider
the globally defined skew-product semiflow
\begin{equation}\label{eq:sk-pr 2}
 \begin{array}{ccl}
 \pi:\R_+\times\Om(g) \times K_0 \times \R^m & \longrightarrow & \Om(g) \times K_0 \times \R^m \\
 (\tau,h,x,y_0) & \mapsto &(h{\cdot}\tau,x,y(\tau,h,x,y_0))\,
\end{array}
\end{equation}
over the base flow on $\Om(g) \times K_0$; and for each $x\in \R^n$ fixed, the skew-product semiflow
 \begin{equation}\label{eq:sk-pr x 2}
 \begin{array}{ccl}
 \pi^x: \R_+\times\Om(g)  \times \R^m & \longrightarrow & \Om(g)  \times \R^m \\
 (\tau,h,y_0) & \mapsto &(h{\cdot}\tau,y(\tau,h,x,y_0))\,
\end{array}
\end{equation}
over the base flow on $\Om(g)$.  There always exists the global attractor for the previous skew-product semiflows. We omit the proof, since it follows the same arguments as those in the proof of Proposition~\ref{prop:existe atractor}. Just note that  each $h\in \Om(g)$ is the limit of a sequence $\{g{\cdot}\tau_k\}_{k\geq 1}$  as $k\to \infty$, for some  $\tau_k\to\infty$.
\begin{prop}\label{prop:existe atractor 2}
Under Assumptions~{\rm \ref{asu:1}~(ii)-(iii)} and~{\rm\ref{asu:3}},  given a compact set $K_0\subset \R^n$, the following statements hold:
\begin{itemize}
  \item[\rm{(i)}] There is a global attractor $\A\subset \Om(g) \times K_0 \times \R^m$ for the skew-product semiflow $\pi$ given in~\eqref{eq:sk-pr 2} which can be written as
\[
\A=\bigcup_{(h,x)\in\Om(g)\times K_0} \{(h,x)\}\times A_{(h,x)}
\]
for the sections $A_{(h,x)}=\{ y\in \R^m\mid (h,x,y)\in \A \}$, and the nonautonomous set $\{A_{(h,x)}\}_{(h,x)\in \Om(g)\times K_0}$ is the pullback attractor.
\item[\rm{(ii)}] For each fixed $x\in K_0$, the skew-product semiflow $\pi^x$ given in~\eqref{eq:sk-pr x 2} has a global attractor $\A^x\subset \Om(g)  \times \R^m$ which can be written as
\[
\A^x=\bigcup_{h\in\Om(g)} \{h\}\times A^x_h
\]
for the sections $A^x_h=\{ y\in \R^m\mid (h,y)\in \A^x \}$, and the nonautonomous set $\{A^x_h\}_{h\in \Om(g)}$ is the pullback attractor of  $\pi^x$.
\end{itemize}
   \end{prop}
\begin{rmk}\label{rmk:asump in Arts}
The hypotheses in this section are equivalent to those in the paper by Artstein~\cite{paper:ZA99} for the case of dynamics in the complete space  $\mathcal{H}\times\R^m$. More precisely, under Assumption~\ref{asu:3}~(i) on the existence of solutions, let us check that  Assumption~\ref{asu:3}~(ii) is equivalent to  Assumption~5.1~(iii) in~\cite{paper:ZA99} with $G(x):=\R^m$ for $x\in\R^n$.

The argument to see that our hypothesis implies that of Artstein's is the same as the one in Remark~\ref{rmk:atractor uniforme}~(1), taking the compact set $K(x):=\bigcup_{h\in\Om(g)} A_h^x\subset \R^m$ for each $x\in \R^n$.
Conversely, suppose that for a family of compact sets $\{K(x)\}_{x\in\R^n}$ of $\R^m$, if $x_k\to x_0$ in $\R^n$,  $y_k\to y_0$ in $\R^m$, $\tau_k\geq 0$ and $\sigma_k\to\infty$, then $y(\sigma_k+\tau_k,g,x_k,\tau_k,y_k)$ converges to the compact set $K(x_0)$; and $K$ has a compact graph over compact sets of $\R^n$. Then, we fix a compact set $K_0\subset \R^n$ and look at the parametric family $y'=g(x,y,\tau)$ for $x\in K_0$. By the compact graph assumption, we can take an $r=r(K_0)>0$ so that
\[
\bigcup_{x\in K_0} K(x) \subset B(0,r)\subset \R^m\,,
\]
where $B(0, r)$ is the ball of $\R^m$ centered at $0$ with radius $ r$. An argument by contradiction (in the same fashion as in Proposition~\ref{prop:unifom ulti bound on K}, but much shorter here) permits to conclude that the solutions are uniformly ultimately bounded on $[0,\infty)$.
%
\end{rmk}

We now state, for completeness, the  result on the forward attracting properties of either the parametrically inflated pullback attractors or the pullback attractor. Once more, the proof can be omitted. Let us first rephrase in this new context the additional assumptions of continuity introduced in Section \ref{secSlowfast}.
\begin{assu}\label{assu:cont A^x 2}
The setvalued mapping $x\in K_0\mapsto \A^x\subset \Om(g)\times \R^m$ is continuous for the Hausdorff distance $d_H$.
\end{assu}

\begin{assu}\label{assu:cont A_(h,x) 2}
The setvalued mapping $(h,x)\in \Om(g)\times K_0\mapsto A_{(h,x)}\subset \R^m$ is continuous for the Hausdorff distance $d_H$.
\end{assu}

\begin{thm}\label{teor:param inflated pullb att 2}
Under Assumptions~{\rm \ref{asu:1}~(ii)-(iii)} and~{\rm\ref{asu:3}},  given a compact set $K_0\subset \R^n$, let $\A\subset \Om(g) \times K_0 \times \R^m$  be the global attractor  for the skew-product semiflow~\eqref{eq:sk-pr 2} and let $\A^x\subset \Om(g) \times \R^m$  be the global attractor  for the skew-product semiflow~\eqref{eq:sk-pr x 2} for each fixed $x\in K_0$. Then, fixed a $\delta>0$,  for every bounded set $D\subset \R^m$,
\begin{equation*}
  \lim_{\tau\to\infty} \bigg( \sup_{h\in \Om(g),\, x\in K_0} {\rm dist}\big(y(\tau,h,x,D),A_{(h{\cdot}\tau,x)}[\delta]\big) \bigg) =0\,.
\end{equation*}
If Assumption~{\rm\ref{assu:cont A^x 2}} holds, then:
\begin{equation*}
  \lim_{\tau\to\infty} \bigg( \sup_{h\in \Om(g),\, x\in K_0} {\rm dist}\big(y(\tau,h,x,D),A^x_{h{\cdot}\tau}[\delta]\big) \bigg) =0\,.
\end{equation*}
Finally, if Assumption~{\rm\ref{assu:cont A_(h,x) 2}} holds,  then:
\begin{equation*}
  \lim_{\tau\to\infty} \bigg( \sup_{h\in \Om(g),\, x\in K_0} {\rm dist}\big(y(\tau,h,x,D),A_{(h{\cdot}\tau,x)}\big) \bigg) =0\,.
\end{equation*}
\end{thm}
At this point we state the main theorem in this section, proving that the solutions $y_{\ep_k}(\tau)$ track  the fibers of the parametrically inflated pullback attractor along arcs of length $T$ of trajectories in the omega-limit set $\Om(g)$ and the solution $x_0(\ep_k\tau)$.
\begin{thm}\label{teor:main omega-lim}
Under Assumptions~{\rm \ref{asu:1}} and~{\rm\ref{asu:3}}, for each $\ep>0$  let $(x_\ep(t),y_\ep(t))$ be the solution of~\eqref{eq:slowfastent}-\eqref{eq:initial values}. Then, given $t_0>0$, for $\ep$ small enough, $(x_\ep(t),y_\ep(t))$ can be extended to $t\in [0,t_0]$. Furthermore:
\begin{itemize}
\item[\rm{(i)}]  Every sequence $\ep_j\to 0$ has a subsequence, say $\ep_k\downarrow 0$, such that $x_{\ep_k}(t)$ converges uniformly for $t\in [0,t_0]$ to a solution of the differential inclusion~\eqref{eq:diff incl}.
Let the limit be $x_0(t)$, for $t\in [0,t_0]$.
\item[\rm{(ii)}] For the fast variables in the fast timescale $y_{\ep_k}(\tau)$, the behaviour is as follows. For a fixed small $r>0$, let $K_0$ be the compact tubular set defined in~\eqref{eq:K0} and
let $\A\subset \Om(g) \times K_0\times \R^m$  be the global attractor  for the skew-product semiflow~\eqref{eq:sk-pr 2}. Then, given $\delta >0$, there exist times $T>0$, $\tau_0>0$, a constant $0<\tilde \delta\leq \delta$, and an integer $k_0>0$ with $\tau_0+2T<t_0/\ep_{k_0}$ such that for every sequence of maps $\{h_i\}_{i\geq 0}\subset \Om(g)$ with $d\big(g{\cdot}(\tau_0+iT),h_i\big)<\tilde\delta$  for $i\geq 0$, for every $k\geq k_0$
and for every  $\tau\in[\tau_0+(i+1)T,\tau_0+(i+2)T]\cap [\tau_0+T,t_0/\ep_k]$ $($if any$)$, it holds that
\begin{equation}\label{eq:teorema 2}
 {\rm dist}\big(y_{\ep_k}(\tau),A_{(h_i{\cdot}(\tau-\tau_0-iT),x_0(\ep_k\tau))}[\delta]\big)\leq  \delta\,.
\end{equation}
\end{itemize}
\end{thm}
\begin{proof}
First of all,  with the assumptions made we can apply Theorem~5.3 in~\cite{paper:ZA99} (see Remark~\ref{rmk:asump in Arts}). Then, (i) follows, as it has been explained in the proof of Theorem~\ref{teor:main}.

To prove (ii), the argumentation is also basically the same as in the proof of Theorem~\ref{teor:main}, but there are some additional technical details which should be made clear. Given $0<\delta<1$, the value of the time $T$ is obtained from Theorem~\ref{teor:param inflated pullb att 2} applied to a compact set $D\subset\R^m$. This time the set $D$ should  contain the whole set $\{y(\tau,g,x,y_0)\mid \tau\geq 0,\, x\in K_0\}$ which is bounded because of Assumption~\ref{asu:3} and, roughly speaking, a neighbourhood of radius $1$ of the set
\begin{equation*}
  \bigcup_{(h,x)\in \Om(g)\times K_0} A_{(h,x)} \subset  \R^m\,,
\end{equation*}
so that, whenever $z$ is less than $1$ away from this set, then $z\in D$. Then, for $\tau\geq T$,
\begin{equation}\label{eq:unif atract 3}
  \sup_{h\in \Om(g),\, x\in K_0} {\rm dist}\big(y(\tau,h,x,D),A_{(h{\cdot}\tau,x)}[\delta]\big)  < \frac{\delta}{3}\,.
\end{equation}
Next, for  $\delta>0$, $T>0$ and $D$, we apply the uniform continuity of the cocycle map $y$ on the compact subset $[T,2T]\times \mathcal{H}\times K_0\times D$ to find a $0<\tilde \delta\leq \delta$ such that, if $h,\tilde h\in\mathcal{H} $ are such that  $d(h,\tilde h)<\tilde \delta$, then
\begin{equation}\label{eq:1}
  \big|y(\tau,h,x,z)-y(\tau,\tilde h,x,z)\big|<\frac{\delta}{3}\quad\hbox{for all}\; \tau\in  [T,2T]\,,\;x\in K_0\,,\; z\in D\,.
\end{equation}
Associated to this value of $\tilde \delta>0$, there is a time $\tau_0>0$ such that ${\rm dist}\big(g{\cdot}\tau,\Om(g) \big)<\tilde\delta$ for every $\tau\geq \tau_0$. This follows easily from the definition of the omega-limit set, arguing by contradiction. As a consequence, there exist  sequences of maps $\{h_i\}_{i\geq 0}\subset \Om(g)$ with $d\big(g{\cdot}(\tau_0+iT),h_i\big)<\tilde\delta$ for each $i\geq 0$. Let $\{h_i\}_{i\geq 0}\subset \Om(g)$ be any of them.

The next step is to consider the compact subset
\begin{equation*}
  K:=y\big([0,\tau_0+2T]\times \mathcal{H}\times K_0\times D \big) \subset \R^m
\end{equation*}
and associated to it, a Lipschitz constant $L>0$ such that for the map $G$ in~\eqref{eq:G def}, $
      |G(x,y_1,h{\cdot}\tau)-G(x,y_2,h{\cdot}\tau)|\leq L\,|y_1-y_2|$
for every $h\in\mathcal{H}$, $y_1,y_2\in K$, $x\in K_0$ and $\tau\in \R$. Then, $\sigma>0$ is chosen so that
\begin{equation*}
\des\frac{\sigma}{L}\,\big(e^{L(\tau_0+2T)}-1\big)\leq \des\frac{\delta}{3}\,,
\end{equation*}
and associated to it, we find an integer $k_0>0$ with  $\tau_0+2T<t_0/\ep_{k_0}$ and so that for every $k\geq k_0$, $x_{\ep_k}(t)\in K_0$ for every $t\in [0,t_0]$ and
\begin{equation}\label{eq:G 2}
  \big|G(x_{\ep_k}(\ep_k\tau),y,h)-G(x_0(\ep_k s),y,h)\big|\leq \sigma
\end{equation}
for every $\tau,s\in [0,t_0/\ep_k]$ with $|\tau-s|\leq \tau_0+2T$,  $y\in K$ and $h\in\mathcal{H}$.

At this point the iterative process starts. In the first step,  relation~\eqref{eq:teorema 2} is proved with $i=0$, i.e., with $h_0\in \Om(g)$ and for $\tau\in [\tau_0+T,\tau_0+2T]$. To do it, for each fixed $s\in [\tau_0+T,\tau_0+2T]$, $y_{\ep_k}(s)$ is approximated by the value at $\tau=s$ of the solution $y(\tau,g,x_0(\ep_k s),y_0)$ integrated on the interval $[0,\tau_0+2T]$. By  construction, this solution lies in $K$ and using~\eqref{eq:G 2} we can  apply Lemma~\ref{lem:comparacion} to get that
\[
\big|y_{\ep_k}(s)-\left.y(\tau,g,x_0(\ep_k s),y_0)\right|_{\tau=s}\big|\leq  \des\frac{\delta}{3}\quad\hbox{for all}\;s\in [\tau_0+T,\tau_0+2T]\,.
\]
Now, for $\tau\in [\tau_0+T,\tau_0+2T]$, by the cocycle identity, we rewrite $y(\tau,g,x_0(\ep_k s),y_0)=y(\tau-\tau_0,g{\cdot}\tau_0,x_0(\ep_k s),y(\tau_0,g,x_0(\ep_k s),y_0) )$. Since we have that $\tau-\tau_0\in [T,2T]$,  $y(\tau_0,g,x_0(\ep_k s),y_0)\in D$ by the choice of $D$, and $d(g{\cdot}\tau_0,h_0)<\tilde\delta$, relation~\eqref{eq:1} applies when we compare with $y(\tau-\tau_0,h_0,x_0(\ep_k s),y(\tau_0,g,x_0(\ep_k s),y_0) )$, so that for all $s\in [\tau_0+T,\tau_0+2T]$,
\[
\big| \left.y(\tau,g,x_0(\ep_k s),y_0)\right|_{\tau=s} - \left. y(\tau-\tau_0,h_0,x_0(\ep_k s),y(\tau_0,g,x_0(\ep_k s),y_0))\right|_{\tau=s}\big|<\frac{\delta}{3}\,.
\]
Finally, since $h_0\in \Om(g)$ and $\tau-\tau_0\geq T$, we can use~\eqref{eq:unif atract 3} to complete the proof of~\eqref{eq:teorema 2} on $[\tau_0+T,\tau_0+2T]$ by the triangle inequality.

Before we proceed to the second step, we note that in particular
\[
{\rm dist}\big(y_{\ep_k}(\tau_0+T),A_{(h_0{\cdot}T,x_0(\ep_k(\tau_0+T)))}[\delta]\big)\leq  \delta\,.
\]
By the choice of $D$, this implies that $y_{\ep_k}(\tau_0+T)\in D$ and we can reproduce the same arguments as before for each fixed  $s\in [\tau_0+2T,\tau_0+3T]$,  approximating $y_{\ep_k}(s)$ by the value at $\tau=s$ of the solution $y(\tau,g,x_0(\ep_k s),\tau_0+T,y_{\ep_k}(\tau_0+T))$ integrated on the interval $[\tau_0+T,\tau_0+3T]$, and  using this time the map $h_1\in \Om(g)$ that satisfies $d(g{\cdot}(\tau_0+T),h_1)<\tilde\delta$. At this point, it seems unnecessary to give more details of the proof. The proof is finished.
\end{proof}
\begin{rmk}
  In the same situation as in the previous theorem, in the particular case when the orbit of $g$ is asymptotic to an orbit inside $\Om(g)$, i.e., when there is $h_0\in \Om(g)$ such that $\lim_{\tau\to\infty} d(g{\cdot}\tau,h_0{\cdot}\tau)=0$, one can simply follow the orbit of $h_0$ for the tracking.
  \end{rmk}
We now give an analog of  Theorem~\ref{teor:valor final} in the more general setting of this section. Given  a compact set $K_0\subset \R^n$ and a  map $h\notin \Om(g)$ with $\rho={\rm dist}\big(h,\Om(g) \big)>0$, let us introduce  for $x\in K_0$ and $\delta\geq \rho$, the fiber sets
\[
B_{(h,x)}[\delta]:=\bigcup_{(\tilde h,\tilde x)\in \Theta_{(h,x,\delta)}} A_{(\tilde h,\tilde x)}
\]
for the set $\Theta_{(h,x,\delta)}=\big\{(\tilde h,\tilde x)\in \Om(g)\times K_0  \mid  d(\tilde h, h)\leq \delta,\,|\tilde x-x|\leq \delta\big\}$.
\begin{thm}\label{teor:valor final 2}
Under Assumptions~{\rm \ref{asu:1}} and~{\rm\ref{asu:3}}, suppose  that $x_{\ep_k}(t)$ converges to $x_0(t)$  uniformly on $[0,t_0]$ and let the compact set $K_0$ and the global attractor $\A$ be the ones determined in Theorem~{\rm \ref{teor:main omega-lim}}.
For each integer $k\geq 1$, let
\begin{align*}
 \delta_k:=\inf\big\{\delta>0 \mid\;  &{\rm dist}\big(g{\cdot}(t_0/\ep_k),\Om(g)\big)\leq \delta\;\\&\hbox{and}\;\, {\rm dist}\big(y_{\ep_k}(t_0/\ep_k),B_{(g{\cdot}(t_0/\ep_k),x_0(t_0))}[\delta]\big)\leq  \delta  \big\}\,.
\end{align*}
 Then:
 \begin{itemize}
   \item[\rm{(i)}] $\des\lim_{k\to \infty} \delta_k =0$;
   \item[\rm{(ii)}] $\des\lim_{k\to \infty}  {\rm dist}\big(y_{\ep_k}(t_0/\ep_k),B_{(g{\cdot}(t_0/\ep_k),x_0(t_0))}[\delta_k]\big) =0$;
 \item[\rm{(iii)}] if $g{\cdot}(t_0/\ep_k)\to h_0$ in $(\mathcal{H},d)$ as $k\to\infty$, then $h_0\in \Om(g)$ and
 \begin{equation*}
   \lim_{k\to \infty}  {\rm dist}\big(y_{\ep_k}(t_0/\ep_k),A_{(h_0,x_0(t_0))}\big) =0\,.
 \end{equation*}
 \end{itemize}
\end{thm}
\begin{proof}
To prove (i), fixed a $\delta>0$,  the application of Theorem~\ref{teor:main omega-lim} (ii) to $\delta/2$ provides us with values of $T>0$, $\tau_0>0$, $0<\tilde\delta\leq \delta/2$ and an integer $k_0\geq 1$. Note that we can additionally take the constant $\tilde \delta$  small enough so that, if $d(h,\tilde h)<\tilde\delta$, then $d(h{\cdot}\tau,\tilde h{\cdot}\tau)<\delta/2$ for every $\tau\in [T,2T]$, by the uniform continuity of the flow map on $[T,2T]\times\mathcal{H}$.
Then, for every $k\geq k_0$ let us take $i=i(k)$ such that $t_0/\ep_k\in [\tau_0+(i+1)T,\tau_0+(i+2)T]$ (so that $t_0/\ep_k-\tau_0-iT\in [T,2T]$). If $h_i\in \Om(g)$ satisfies    $d\big(g{\cdot}(\tau_0+iT),h_i\big)<\tilde\delta$, then
${\rm dist}\big(y_{\ep_k}(t_0/\ep_k),A_{(h_i{\cdot}(t_0/\ep_k-\tau_0-iT),x_0(t_0))}[\delta/2]\big)\leq  \delta/2$.
At this point just note that $d(g{\cdot}(t_0/\ep_k),h_i{\cdot}(t_0/\ep_k-\tau_0-iT))<\delta/2$, so that it is easy to check that
\[
A_{(h_i{\cdot}(t_0/\ep_k-\tau_0-iT),x_0(t_0))}[\delta/2]\subseteq B_{(g{\cdot}(t_0/\ep_k),x_0(t_0))}[\delta]
\]
from where it follows that ${\rm dist}\big(y_{\ep_k}(t_0/\ep_k), B_{(g{\cdot}(t_0/\ep_k),x_0(t_0))}[\delta]\big)\leq  \delta/2<\delta$, that is, $\delta_k\leq \delta$ for all $k\geq k_0$ and the proof of (i) is finished. The limit in (ii) is immediate.

To finish, to prove the statement in (iii), argue as in the proof of Theorem~\ref{teor:valor final} (iii) by using the triangle inequality and an argument by contradiction to prove the desired formula. The proof is finished.
\end{proof}

\begin{rmk}
When either Assumption~{\rm\ref{assu:cont A^x 2}} or Assumption~{\rm\ref{assu:cont A_(h,x) 2}} is added in Theorem~\ref{teor:main omega-lim} (ii) (or Theorem~\ref{teor:valor final 2}) for the compact set $K_0$, then the results can be reformulated, offering a more precise description of the zone of attraction, in the line of Theorems~\ref{teor:main 2} (and its Corollary~\ref{coro 1}) and~\ref{teor:main 3}, respectively. We omit the precise statements not to make the paper too long.
\end{rmk}

\end{document}